\documentclass{article}
\pdfoutput=1
\usepackage[a4paper,top=2cm,bottom=2cm,left=3cm,right=3cm,marginparwidth=1.75cm]{geometry}

\usepackage[utf8]{inputenc}
\usepackage{amsfonts}
\usepackage{amsmath,amscd,amsthm,amssymb,mathrsfs,setspace}
\usepackage{mathtools}
\usepackage{cite}
\usepackage{color}
\usepackage{algorithm}
\usepackage{algorithmicx}
\usepackage{algpseudocode}
\usepackage{booktabs}
\usepackage{multirow}
\usepackage{makecell}
\usepackage[colorlinks]{hyperref}
\usepackage[nameinlink,capitalize]{cleveref}
\usepackage{graphicx}
\usepackage{tikz}
\usepackage{url}
\usepackage{fancyhdr}
\usepackage{authblk}
\usepackage{enumitem}
\usepackage{adjustbox}
\usepackage{subcaption}
\usepackage{xcolor, soul}
\sethlcolor{pink}
\usetikzlibrary{patterns,patterns.meta}
\definecolor{lccx}{HTML}{92268F}

\ifpdf
\hypersetup{
	pdftitle = {Block Total Variation Integer Control}, 
	pdfauthor={R. J. Baraldi and P. Manns},
	linkcolor=lccx	
}

\newtheorem{theorem}{Theorem}[section]
\newtheorem{proposition}[theorem]{Proposition}
\newtheorem{lemma}[theorem]{Lemma}

\newtheorem{definition}[theorem]{Definition}
\newtheorem{assumption}[theorem]{Assumption}

\theoremstyle{remark}
\newtheorem{remark}[theorem]{Remark}

\crefname{theorem}{Theorem}{Theorems}
\Crefname{theorem}{Theorem}{Theorems}
\crefname{assumption}{Assumption}{Assumptions}
\Crefname{assumption}{Assumption}{Assumptions}
\crefname{lemma}{Lemma}{Lemmas}
\Crefname{lemma}{Lemma}{Lemmas}
\crefname{definition}{Definition}{Definitions}
\Crefname{definition}{Definition}{Definitions}
\crefname{proposition}{Proposition}{Propositions}
\Crefname{proposition}{Proposition}{Propositions}
\crefname{algorithm}{Algorithm}{Algorithms}
\Crefname{algorithm}{Algorithm}{Algorithms}
\crefname{section}{Section}{Sections}
\Crefname{section}{Section}{Sections}
\crefname{appendix}{Appendix}{Appendices}
\Crefname{appendix}{Appendix}{Appendices}

\newcommand{\inner}[2]{\left(#1, #2\right)}

\DeclareMathOperator*{\argmax}{arg\,max}

\DeclareMathOperator*{\dist}{dist}

\DeclareMathOperator*{\TV}{TV}

\DeclareMathOperator*{\pred}{pred}
\DeclareMathOperator*{\ared}{ared}

\DeclareMathOperator*{\BV}{BV}
\DeclareMathOperator*{\BVW}{BV_W}

\DeclareMathOperator*{\TVB}{TV_B}
\DeclareMathOperator*{\supp}{supp}
\DeclareMathOperator*{\dvg}{div}

\DeclareMathOperator*{\interior}{int}

\newcommand{\N}{\mathbb{N}}
\newcommand{\R}{\mathbb{R}}
\newcommand{\Z}{\mathbb{Z}}
\newcommand{\Ha}{\mathcal{H}}
\newcommand{\conv}{\operatorname{conv}}

\newcommand{\norm}[1]{\left\| #1 \right\|}

\newcommand{\weakstarto}{\stackrel{\ast}{\rightharpoonup}}
\newcommand{\bdvg}[1]{\dvg{}_{{\hspace*{-2pt}#1\hspace*{2pt}}}}
\newcommand{\mres}{\llcorner}
\newcommand{\calD}{\mathcal{D}}
\newcommand{\calA}{\mathcal{A}}
\newcommand{\calW}{\mathcal{W}}
\newcommand{\symdif}{\vartriangle}

\newcommand*\dd{\mathop{}\!\mathrm{d}}

\crefname{assumption}{Assumption}{Assumptions}
\Crefname{assumption}{Assumption}{Assumptions}
\crefname{lemma}{Lemma}{Lemmas}
\Crefname{lemma}{Lemma}{Lemmas}
\crefname{definition}{Definition}{Definitions}
\Crefname{definition}{Definition}{Definitions}
\crefname{proposition}{Proposition}{Propositions}
\Crefname{proposition}{Proposition}{Propositions}
\crefname{algorithm}{Algorithm}{Algorithms}
\Crefname{algorithm}{Algorithm}{Algorithms}

\definecolor{darkgreen}{rgb}{0,0.5,0}
\definecolor{darkorange}{rgb}{0.9,0.5333333333333333,0.35}
\newcommand{\violet}[1]{{\color{violet} #1}}
\newcommand{\blue}[1]{{\color{blue} #1}}
\newcommand{\darkgreen}[1]{{\color{darkgreen} #1}}
\newcommand{\darkorange}[1]{{\color{darkorange} #1}}

\title{Domain Decomposition for Integer Optimal Control with Total Variation Regularization\footnote{
R. Baraldi acknowledges funding by the U.S. Air Force Office of Scientific Research no. \ FA9550-22-1-0248 and Sandia Laboratory Directed Research and Development John von Neumann Fellowship.
P. Manns acknowledges funding by Deutsche Forschungsgemeinschaft (DFG) under project no.\ 515118017.}}

\author[1]{Robert Baraldi}
\affil[1]{Sandia National Laboratories, Optimization and Uncertainty Quantification. P.O.~Box 5800, Albuquerque, New Mexico, 87125, USA, \texttt{rjbaral@sandia.gov}}
\author[2]{Paul Manns}
\affil[2]{Faculty of Mathematics, TU Dortmund University, 44227 Dortmund, Germany, \texttt{paul.manns@tu-dortmund.de}}

\begin{document}
\maketitle
\begin{abstract}
Total variation integer optimal control problems admit solutions and necessary optimality conditions 
via geometric variational analysis. In spite of the existence of said solutions, 
algorithms which solve the discretized objective suffer 
from high numerical cost associated with the combinatorial nature of integer programming. 
Hence, such methods are often limited to small- and medium-sized
problems.

We propose a globally convergent, coordinate descent-inspired algorithm that allows tractable subproblem solutions restricted
to a partition of the domain. Our decomposition method solves relatively small trust-region subproblems that modify the control
variable on a subdomain only.  Given nontrivial subdomain overlap, 
we prove that a global first-order necessary optimality condition is equivalent to a first-order
necessary optimality condition per subdomain. 
We additionally show that sufficient decrease is achieved on a single subdomain by way of a trust-region subproblem solver
using geometric measure theoretic arguments, which we integrate with a greedy patch selection to prove convergence of our algorithm.
We demonstrate the practicality of our algorithm on a benchmark large-scale, PDE-constrained integer optimal control problem, and find
that our method is faster than the state-of-the-art.
\end{abstract}



\section{Introduction}
Let $\Omega \subset \R^d$ be a bounded Lipschitz domain, $d \in \N$, $\alpha > 0$, and $M \in \N$.
We consider integer optimal control problems \cite{hante2013relaxation} of the form
\begin{gather}\label{eq:p}
\begin{aligned}
\min_{w \in L^1(\Omega)}\ & J(w) \coloneqq F(w) + \alpha \TV(w) \\
\text{s.t.}\quad 
& w(x) \in W \coloneqq \{w_1,\ldots,w_M\} \subset \Z
\text{ for almost every (a.e.) }
x \in \Omega,
\end{aligned}\tag{P}
\end{gather}
where $F : L^1(\Omega) \to \R$ is lower semicontinuous with respect to convergence in $L^p(\Omega)$ for some $p \ge 1$ and
bounded below and $\TV : L^1(\Omega) \to [0,\infty]$ denotes the total variation seminorm. Typically, $F = j \circ S$ for
$j$ some tracking-type functional and $S$ a solution operator of a PDE, ODE, or another integral operator.
Such situations arise in control of switched systems \cite{de2019mixed,kaya2020optimal,sager2021mixed},
topology optimization \cite{clason2018total,clason2021convex,haslinger2015topology,leyffer2021convergence,sigmund2013topology},
or optimal experimental design \cite{sager2013sampling,yu2021multidimensional}. 

Specifically, we consider the setting present in \cite{manns2023on},
where the authors proposed to solve a sequence of subproblems
\begin{gather}\label{eq:tr}
\text{{\ref{eq:tr}}}(\bar{w}, g, \Delta) \coloneqq
\left\{
\begin{aligned}
\min_{w \in L^2(\Omega)}\ & (g, w - \bar{w})_{L^2(\Omega)} + \alpha \TV(w)-\alpha \TV(\bar{w})\\
\text{s.t.}\quad & \|w - \bar{w}\|_{L^1(\Omega)} \le \Delta\text{ and }w(x) \in W \text{ for a.e.\ } x \in \Omega,
\end{aligned}
\right.
\tag{TR}
\end{gather}
within a trust-region algorithm for globalization. Therein, the function $g$ is (an approximation of) the Riesz representative
$\nabla F(\bar{w})\in L^{\infty}(\Omega)$ if $F$ is Fr\'{e}chet differentiable wrt.\ the $L^1$-norm so that the trust-region
subproblem \eqref{eq:tr} arises from \eqref{eq:p} by solving a partially linearized model in an $L^1(\Omega)$-ball around a given point $\bar{w}$,
that is, the current iterate of the algorithm. While we will assume said Fr\'{e}chet differentiability of $F$ to derive
our results, we note that the presented algorithm and its iterations may still be well defined and meaningful if, e.g., subgradients are used if $F$ is
not differentiable. After discretization of the domain $\Omega$ and the introduction of a piecewise constant ansatz for $w$, the trust-region 
subproblems become integer linear programs \cite{leyffer2022sequential};
these are often computationally expensive to solve in practice (see, e.g., 
\cite{leyffer2022sequential,manns2024discrete}).
This difficulty stems from the large number of variables present in the
integer linear programs, causing long running times or
even subproblem computational infeasibility when particularly 
fine discretizations are chosen. While dynamic programming-based
algorithms \cite{marko2023integer,severitt2023efficient}
allow for efficient subproblem solvers on one-dimensional
domains, no such approach is known for multi-dimensional domains.
Instead, recent results \cite{manns2024discrete}
indicate that the problems are likely NP-hard, thereby requiring
integer programming solver-based methods. 
These may moderately improve run times,
but such structure-exploitation techniques are
unlikely to decrease wallclock time by an appreciable amount. 
Moreover, when considering a piecewise control function ansatz,
as in \cite{manns2023on,manns2024discrete},
the \emph{exact} total variation is the sum of interface lengths between
the level sets weighted by the jump heights across the sets; see, for example,
\cite{herrmann2019discrete}. Such settings are inevitable 
if the discretized controls are discrete-valued.
Using uniform meshes and driving the mesh size to zero
yields a gap between the discrete problem 
and the intended infinite-dimensional
limit problem, because the restricted geometry of the
discretization is reflected in the infinite-dimensional limit \cite{cristinelli2023conditional}. 
The difficulty is compounded in 
finite difference and finite element-based approximations, where the underlying meshes
for the control and finite-element ansatz coincide \cite{chambolle2021approximating}, 
cannot be applied directly because interpolation and
projection operators can violate integer feasibility.
The authors of \cite{schiemann2024discretization} overcome this problem
via a two-level discretization combined with a cutting plane 
generation strategy, the latter of which successively enriches
the integer subproblems, e.g. \eqref{eq:tr}, 
with linear inequalities. 
These allow isotropic approximations of the total variation seminorm despite the aforementioned geometric restriction. However, this process generates even larger problem formulations that are even less tractable; see the compute times reported in \S6.4 of \cite{schiemann2024integer}, in particular Table 6.11,
where these problem formulations are used within the algorithm 
from \cite{manns2023on}.
Consequently, reducing the size of subproblems is a 
sensible starting point to scale the algorithm from \cite{leyffer2022sequential,manns2023on} to practical
problem sizes, even given the potential cost incurred by computing more subproblem solutions.

Many large-scale problems, particularly in PDE
numerics, are solved using domain decomposition approaches; see the
monographs \cite{chan1994domain,dolean2015introduction,lagnese2012domain,mathew2008domain}. In finite-dimensional optimization, coordinate-descent algorithms \cite{tseng2009coordinate,wright2015coordinate} 
can be interpreted as decomposition-based methods and
are particularly popular for settings with
separable or block-separable objective functions.
Coordinate-descent algorithms update
only on a subset of the coordinates at a time, either cyclically
or via random selection; in extreme cases, 
this may be only one coordinate. 
Hence, cost per iteration is reduced to an acceptable
level although the number of iterations may be very large.
Such methods have been explored for convex optimization with $\TV$ terms, primarily in the context of image
denoising with an $L^2$-fidelity term.
Common splitting methods, such as the Chambolle--Pock algorithm \cite{pock2010tv}, perform well for small- and medium-scale problems and typically do not employ any domain decomposition. 
For larger $\TV$-regularized problems, domain decomposition, or similarly utilized coordinate-descent methods, have been developed for both primal and (pre-)dual formulations; see the overview articles \cite{lee2020recent,Langer2023}. While the setting therein concerns real-valued inputs and is in particular convex, these domain decomposition techniques also make concessions in of terms convergence guarantees or
quality of the result. Specifically, naive coordinate-descent methods generally require a separability condition on the convex, nonsmooth term \cite{tseng2009coordinate}
to converge to global minimizers. This is, however, violated by the $\TV$-term. Alternatively, one can consider the Euler--Lagrange equation for perturbed TV \cite{chen2007multi,Langer2023}, which, however, fails to preserve discontinuities and edges \cite{Langer2023}.
For pre-dual TV formulations \cite{kunisch2004bi}, convergence guarantees to a (global) minimizer have been proven for semismooth-Newton-type or accelerated splitting methods with overlapping and non-overlapping domain decomposition methods \cite{Langer2023,gaspoz2019predual}; additionally, patch subproblems can be solved in parallel \cite{langer2015non,lee2019dual,lee2019fe,lee2019non}.
Discretizations considered include finite-differences \cite{langer2015non} and finite elements \cite{lee2019fe}. The former
can be solved with a variety of algorithms, e.g., Chambolle--Pock \cite{pock2010tv} or semismooth Newton \cite{langer2015non}, yielding
convergence to a solution of the dual problem.
The finite-element splitting approach \cite{lee2019fe}
uses accelerated iterative soft-thresholding to achieve convergence
to a solution of the dual problem for non-overlapping domains, 
allowing for direct parallelization. For primal decomposition 
methods to achieve such a result, a communication mechanism is 
required between or after solving the decomposed problems;
see \cite{gaspoz2019predual}.

Our setting is inherently non-convex, hence we do not have strong duality and therefore follow a primal decomposition approach while striving for convergence to stationary points as in the case sans decomposition \cite{leyffer2022sequential,schiemann2024discretization}.

\subsection{Contribution}
We transfer the idea of coordinate descent to \eqref{eq:p} and the 
trust-region algorithm proposed in \cite{manns2023on}.  We decompose
the domain $\Omega$ into smaller patches and solve the instances
of \eqref{eq:tr} on these patches. By prescribing a covering property
and thus nontrivial overlap for the patches, we can 
localize the first-order necessary optimality condition on \eqref{eq:p} 
from \cite{manns2023on} to a first-order necessary optimality on all
patch problems.
This in turn permits construction of competitor sequences that allow us
to prove sufficient decrease properties and 
determine stopping criteria for patch problem solution tabulation; 
the latter of which is based on the predicted reductions
and trust-region radii. In turn, we are able to show convergence
for a superordinate trust-region algorithm that tabulates solutions to
patch subproblems and makes a greedy update of the iterate.
Hence, there is no theoretical gap when compared to
the asymptotics of the trust-region algorithm shown in \cite{manns2023on}.
We additionally note that this does not contradict the missing
convergence guarantees of primal decomposition methods for the convex 
problems mentioned in the introduction; our optimality condition
is localized on the interfaces of the level sets of $w$,
and thus weaker than stationarity in a real-valued
and differentiable setting.

We have executed the algorithm for a test problem on a one-dimensional domain and a test problem
on a two-dimensional domain. On the one-dimensional domain, where we can use an efficient subproblem
solver from \cite{severitt2023efficient}, the SLIP algorithm scales very well and the subproblem solves
are faster than the solution to state and adjoint equation.
Consequently, the coordinate-based approach does not give a performance
benefit in this case (except for extreme situations). This is in contrast to the two-dimensional problem,
where a sufficiently large number of patches does not impair the quality of the resulting objective function
values but gives high speedups for expensive instances. For the most expensive instance in our benchmark,
our new algorithm returns a point of very similar objective value with a speedup of approximately
$125$ times the approach described \cite{manns2023on}. Consequently, the proposed algorithm
is an important step towards solving large-scale problems, as it can be computationally much
cheaper without impairing the quality of the computed points.

While the proposed greedy update is clearly expensive, our current analysis deems it necessary to avoid situations
like the one demonstrated in the famous example in \cite{powell1973search}, in which a coordinate descent-algorithm circumscribes
the local optimum by traversing the surrounding level-sets.
In such cases, coordinate descent fails to converge 
in $\R^n$ when using a fixed coordinate-selection scheme. Since, to our knowledge no other algorithmic
approaches exist thus far and integer programming problems become otherwise quickly
completely intractable, we believe that a greedy approach is justified.
Additionally, we integrate a heuristic acceleration step into the algorithm that combines
block updates sequentially (largest predicted decrease to smallest) until an {\it a posteriori} decrease condition fails. We also believe that the algorithm has a high 
potential for further improvements in terms of scalability,
in particular by means of randomized patch selection. Such randomization
is usually key to obtain good convergence properties for coordinate-descent algorithms
without tabulation and greedy selection \cite{wright2015coordinate}.

\subsection{Structure of the remainder} We first introduce important notation
in \cref{sec:notation}, and then the domain decomposition
and localized/patch first-order necessary optimality (stationarity)
in \cref{sec:localization}. \Cref{sec:algorithm} introduces the
trust-region algorithm that includes the aforementioned tabulation
and greedy update selection. In \cref{sec:convergence}, we prove
that instationary points lead to finite tabulation and
acceptable patch updates and in turn convergence of the trust-region
algorithm by means of the aforementioned competitor constructions.
\Cref{sec:numerics} provides preliminary computational results.

\section{Notation}\label{sec:notation}
Let $[n] \coloneqq \{1,\ldots,n\}$ for $n \in \N$.
The symmetric difference of sets $A,B\subset \R^d$ is $A\symdif B$.
We denote the Lebesgue measure $\lambda$, and the Lebesgue space $L^p(\Omega)$ on $\Omega$
as $\norm{\cdot}_{L^p(\Omega)}$ by $\norm{\cdot}_{L^p}$ with inner product $\inner{\cdot}{\cdot}_{L^2}$.
The restriction of a measure $\mu$ to set $A$ is $\mu\mres A$.
For a set $A$, the function $\chi_A$ is the $\{0,1\}$-valued characteristic function of $A$.
We call a partition of a set into sets of finite perimeter a \emph{Caccioppoli partition}.
For a set $A \subset \Omega$, we denote its points of density $1$ and $0$ with respect to the
Lebesgue measure by $A^{(1)}$ and $A^{(0)}$, see also \cite[\S5.3]{maggi2012sets}.
For measureable $E\subset \Omega$, the perimeter is defined as in \cite{maggi2012sets, manns2023on}
$$
P(E,\Omega)\coloneqq \sup\left\{\int_E \dvg \varphi(x) \dd x \,\bigg\vert\, \varphi\in C_c^1(\Omega,\R^d), \, \sup_{x\in\Omega}\norm{\varphi(x)}\le 1\right\}.
$$
If $P(E,\Omega)< \infty$, it is a {\it Caccioppoli set}.
A partition $\{E_i\}_{i\in I}$ of $\Omega$ is a Caccioppoli partition
if $\sum_{i\in I}P(E_i,\Omega)< \infty$.
The {\it topological boundary} is defined as $\partial E$ and the {\it reduced
boundary} is $\partial^*E$.
The {\it essential boundary} $\partial^e A$ is the set of points with density of neither
1 nor 0 with respect to $A$.
Note that unless noted otherwise, we consider $\partial^*E$, $\partial^eE$, $E^{(0)}$, $E^{(1)}$,
$\partial E$ with respect to $\Omega$ so that
$P(E,\Omega)  = \Ha^{d-1}(\partial^*E)$,
where $\Ha^{d-1}$ denotes the $d-1$-dimensional Hausdorff measure.
For a given Caccioppoli set $E$, we denote its unit outer normal vector on the reduced boundary
by $n_E$. Moreover, if, in addition, a vector field $\phi \in C_c^\infty(\Omega, \R^d)$ is given, we recall that
its boundary divergence $\bdvg{E_i} \phi : \partial^*E \to \R$ is defined by
$\bdvg{E_i} \coloneqq \dvg \phi - n_E \cdot \nabla \phi n_E$ \cite[\S17.3]{maggi2012sets}.

A function $u\in L^1(\Omega)$ is of bounded variation ($u\in \BV(\Omega)$)
if its distributional derivative $Du$  is a finite Radon measure over $\Omega$, i.e.,
$\TV(u)\coloneqq \vert D u\vert(\Omega)< \infty$ where $\vert u\vert$
is the total variation of measure $\mu$.
For a Borel set $E\subset \Omega$, $\TV(\chi_E) = P(E,\Omega)$.
The fact that a sequence of functions $\{w^n\}_n\subset\BV(\Omega)$ converges
{\it weakly-$^*$} to $w \in \BV(\Omega)$ is denoted by $w^n \weakstarto w$
when $w^n\rightarrow w$ in $L^1(\Omega)$ and $\limsup_{n\rightarrow\infty}\TV(w^n)< \infty$.
A sequence $\{w^n\}_n$ converges {\it strictly} to $w\in \BV(w)$
if $w^n\rightarrow w$ in $L^1(\Omega)$ and $\TV(w^n)\rightarrow \TV(w)<\infty$.

Feasible points of \eqref{eq:p} are functions in $\BV(\Omega)$ that attain
values only in the finite set $W$; additionally, their
distributional derivatives are absolutely continuous with respect
to $\Ha^{d-1}$.
The feasible set of \eqref{eq:p} is $\BVW(\Omega)$ defined by
\[
\BVW(\Omega)\coloneqq \{w \in \BV(\Omega)\,\vert\, w(x) \in W \, \text{for a.e.\ } x\in\Omega\}\subset \BV(\Omega),
\]
which is weak-$^*$ sequentially closed in $\BV(\Omega)$ \cite{manns2023on}.
Note that due to the discreteness restriction, $\BVW(\Omega)$ is a bounded
subset of $L^\infty(\Omega)$. This also implies that the assumed lower semicontinuity of $F$
for some $p \ge 1$ gives lower semicontinuity of $F$ for all $p \in [1,\infty)$ when restricting to $\BVW(\Omega)$.
In order to avoid cumbersome notation, we define for open sets $B \subset \Omega$
the restricted total variation to $B$ as
\[ \TVB(w) \coloneqq
   \sum_{i=1}^{M-1}\sum_{j=i+1}^M |w_i - w_j|\Ha^{d-1} \mres B (\partial^* E_i \cap \partial^* E_j)
\]
for $w \in \BVW(\Omega)$ with corresponding Caccioppoli partition $\{E_1,\ldots,E_M\}$ of $\Omega$
such that $w = \sum_{i=1}^Mw_i \chi_{E_i}$. Clearly, $\TV(w) = \TV_{\Omega}(w)$.

\begin{definition}[Definition 3.1 in \cite{manns2023on}]
\begin{enumerate}
    \item A one parameter family of diffeomorphisms is the
    $f\in C^\infty:(-\epsilon, \epsilon)\times \Omega\rightarrow \Omega$
    for some $\epsilon>0$ such that for all $t\in(-\epsilon,\epsilon)$, the function $f_t(\cdot)\coloneqq f(t,\cdot):\Omega\rightarrow\Omega$
    is a diffeomorphism.
    \item For open $A\subset \Omega$, the family $(f_t)_{t\in(-\epsilon,\epsilon)}$
    is a {\it local variation} in $A$ if, in addition to 1., $f_0(x)=x$ for all
    $x\in \Omega$ and there is a compact set $K\subset A$ such that $\{x\in \R^d \vert f_t(x)\neq x\} \subset K$ for all $t\in (-\epsilon, \epsilon)$.
    \item For a local variation, its {\it initial velocity} is defined as
    $\phi(x)\coloneqq \tfrac{\partial f}{\partial t}(0,x)$ for $x\in \Omega$.
\end{enumerate}
\end{definition}

\section{Relation of Stationarity for \texorpdfstring{\eqref{eq:p}}{TEXT} to Patch Problems}\label{sec:localization}
We propose to solve subproblems on patches that only update the iterate in some part
of the domain $D \subset \Omega$. We briefly recall stationarity for \eqref{eq:p}
from \cite{manns2023on} and subsequently define a restricted \emph{patch-stationarity}
concept to patches $D$ that decompose the domain $\Omega$. Then we prove equivalence of
stationarity and patch-stationarity. This section references several results from
\cite{manns2023on}, which is written under the general assumption $d \ge 2$. 
Inspecting the arguments of \cite{manns2023on}, we observe that $d \ge 2$
is not required for the proofs of the referenced results (it is required in \cite{manns2023on}
for proving Theorem 5.2 therein and results building on it) so that the referenced results
can all safely be used here in our unified setting $d \ge 1$.

\subsubsection*{Stationarity for \eqref{eq:p}}
We provide the stationarity concept from Definition 4.4 in \cite{manns2023on}, which leans on
stationarity for the prescribed mean curvature problem below.
\begin{definition}[Stationarity, Definition 4.4 in \cite{manns2023on}]\label{dfn:stationarity}
Let $F : L^1(\Omega) \to \R$ be continuously
Fr\'{e}chet differentiable.
Let $w \in \BVW(\Omega)$, that is $w = \sum_{i=1}^M w_i \chi_{E_i}$
for some Caccioppoli partition $\{E_1,\ldots,E_M\}$  of
$\Omega$. Let $\nabla F(w) \in C(\bar{\Omega})$.
Then, $w$ is \emph{stationary} if
\begin{gather}\label{eq:stationarity}
\sum_{i=1}^{M - 1} \sum_{j=i + 1}^M
\int_{\partial^*{E}_i \cap \partial^* E_j}
(w_j - w_i)\nabla F(w)(x)\phi(x)\cdot n_{E_i}(x)
	- \alpha |w_i - w_j| \bdvg{E_i} \phi(x)
	\dd \Ha^{d-1}(x) = 0
\end{gather}
holds for all $\phi \in C^\infty_c(\Omega, \R^d)$.
\end{definition}
Stationarity as defined above is a first-order necessary optimality condition for \eqref{eq:p}
that arises from a first-order variation of the objective functional at a locally optimal point
with respect to perturbations of a feasible point's level sets;
see \cref{prp:first-order-opt} below, which is shown in Theorem 4.6 in \cite{manns2023on}.
It is an extension of the variational first-order optimality for the \emph{prescribed
mean curvature problem} from geometric measure theory; see, e.g., \cite[Equation 12.32 and Remark 17.11]{maggi2012sets}. Moreover,
\cite{leyffer2022sequential,manns2023on} prove
the limit points of a trust-region algorithm in $\BVW(\Omega)$ are stationary under 
regularity assumptions, the latter of which may be interpreted as \emph{constraint qualifications} in analogy
to classical nonlinear programming theory. Due to the non-convexity, \eqref{eq:stationarity} is a necessary and
not a sufficient optimality condition.

The variational formulation \eqref{eq:stationarity} can be interpreted as follows:
On $\partial^*E_i\cap \partial^* E_j$, the partition $E_i$
has distributional mean curvature $-\nabla F(w) \frac{w_i - w_j}{\alpha\vert w_i - w_j\vert}$.
We refer to \cite{manns2023homotopy,manns2023on} for more details and present the following proposition:
\begin{proposition}[Theorem 4.6 in \cite{manns2023on}, Proposition 2.4 in \cite{manns2023homotopy}]\label{prp:first-order-opt}
Let $F : L^1(\Omega) \to \R$ and $\bar{w} \in \BVW(\Omega)$
satisfy the assumptions of \cref{dfn:stationarity}.
If there is $r > 0$ such that
\[ F(\bar{w}) + \alpha \TV(\bar{w}) \le F(w) + \alpha \TV(w)
\]
holds for all $w \in \BVW(\Omega)$ such that
$\|w - \bar{w}\|_{L^1} \le r$, then $\bar{w}$ is stationary.
\end{proposition}
From the stationarity equality \eqref{eq:stationarity}
and its derivation by means of local variations,
it is clear that the above-defined stationarity concept
is based on purely local information that is concentrated
on the interfaces between the level sets of the
feasible point $w \in \BVW(\Omega)$ under consideration,
see the analysis of stationarity and corresponding remarks
in \cite{leyffer2022sequential,manns2023homotopy,manns2023on}.
In particular, the variational \emph{for all} character of the condition
in \cref{dfn:stationarity} implies that constant functions, which do not have level set boundaries inside $\Omega$,
are always stationary. We stress that this does not mean that the algorithm proposed in
\cite{manns2023on} and the algorithm analyzed in this work necessarily stop
at constant functions. On the contrary, we initialize our algorithm with a constant function in our experiments
in \cref{sec:numerics} and can observe that it produces a very different point.
However, by choosing a very large value of $\alpha > 0$, it is possible to construct
results where the algorithm cannot leave the stationary initial point; see also Example 3
in \cite{Langer2023}.

\subsubsection*{Patch-stationarity for \eqref{eq:p}}
We now assume a family of patches $\calD$ that cover our computational domain $\Omega$.
We want to relate the concept of stationarity from the patches to the whole domain
and vice versa. Therefore, we introduce and analyze a patch-based stationarity
concept below.

\begin{assumption}\label{ass:calD_open_cover}
Let $\calD \subset 2^{\Omega}$ be a finite, open cover of $\Omega$.
\end{assumption}

\begin{definition}[Patch-stationarity with respect to $\calD$]\label{dfn:patch_stationarity}
Let $\calD \subset 2^{\Omega}$, $F : L^1(\Omega) \to \R$, and $\bar{w} \in \BVW(\Omega)$
satisfy the assumptions of \cref{dfn:stationarity} and \cref{ass:calD_open_cover}.
Then, $\bar{w}$ is \emph{patch-stationary with respect to $\calD$} if for all $D \in \calD$
the identity \eqref{eq:stationarity}
holds for all $\phi \in C^\infty_c(D, \R^d)$.
\end{definition}

We are ready to prove our main result for the localization
of stationarity to the patch problems, namely that
stationarity and local stationarity
with respect to $\calD$ are equivalent.

\begin{theorem}\label{thm:stationarity_localization}
Let $\calD \subset 2^{\Omega}$, $F : L^1(\Omega) \to \R$, and $\bar{w} \in \BVW(\Omega)$
satisfy the assumptions of \cref{dfn:patch_stationarity}.
Then $\bar{w}\in\BVW(\Omega)$ is stationary
if and only if it is patch-stationary
with respect to $\calD$.
\end{theorem}
While the proof of the forward implication is straightforward, the reverse implication requires
some preparatory work. Namely, for an arbitrary but fixed Caccioppoli partition of $\Omega$
and $\phi \in C_c^\infty(\Omega,\R^d)$, we assert the existence of a countable set of disjoint,
closed balls $\{ \overline{B_k}\,|\, k \in \N \}$ so that each $\overline{B_k}$ is contained
in at least one of the patches. Moreover, the $\overline{B_k}$ exhaust $\supp \phi$
except for a set of $\Ha^{d-1}$-measure zero and their boundaries intersect the interfaces of
the Caccioppoli partition only in a set of $\Ha^{d-1}$-measure zero.
This auxiliary result is proven below in \cref{lem:cover} as a direct consequence of the Vitali--Besicovitch
covering theorem.
\begin{lemma}\label{lem:cover}
Let $\calD$ satisfy \cref{ass:calD_open_cover}, $\{E_1,\ldots,E_M\}$ be a Caccioppoli partition
of $\Omega$, and $\phi \in C_c^\infty(\Omega,\R^d)$. Then there is a countable set
$\tilde{\mathcal{F}} = \{\overline{B_k}\,|\, k \in \N\}$
of pairwise disjoint and closed balls such that
\begin{align}
\overline{B_k} \subset \subset D \text{ for some } D \in \calD \text{ for all } k &\in \N,\\
\Ha^{d-1}\mres \left(\bigcup_{i=1}^M \partial^* E_i\right)
\left(\supp \phi \setminus \bigcup_{k=1}^\infty \overline{B_k} \right) &= 0, \text{ and}\\
\Ha^{d-1}\mres \left(\bigcup_{i=1}^M \partial^* E_i\right)
\left(\partial \overline{B_k}\right) &= 0.\label{eq:no_Hdm1_contribution_of_ball_boundaries}
\end{align}
\end{lemma}
\begin{proof}
The claim follows from the Vitali--Besicovitch covering theorem, see
Theorem 2.19 in \cite{ambrosio2000functions}, if the (uncountable) set of
closed balls
\begin{gather}\label{eq:fine_cover_patch_stationarity}
\mathcal{F} = \left\{ \overline{B_s(x)} \,\middle|\, x \in \supp \phi,
0 < s,
\overline{B_s(x)} \subset D,
D \in \calD,
\text{ and }
\Ha^{d-1}\Bigg(\partial \overline{B_s(x)} \cap \bigcup_{i = 1}^M \partial^* E_i\Bigg) = 0
\right\}
\end{gather}
is a fine cover of $\supp \phi$. Specifically, we need to show (A) that
for all $x \in \supp \phi$ there is a ball $B_{r_x}(x)$, $r_x > 0$, which is (compactly)
contained in at least one patch $D \in \calD$ and (B) that each of these balls contains
infinitely many balls $\overline{B_{s_k}(x)}$ of radii $0 < s_k < r_x$
with $s_k \searrow 0$ that satisfy
$\Ha^{d-1}\Big(\partial \overline{B_{s_k}(x)} \cap \bigcup_{i = 1}^M \partial^* E_i\Big) = 0$.

Claim (A) follows directly from \cref{ass:calD_open_cover}.
Claim (B) follows from the fact that
for balls $B_s(x)$ we have $\Ha^{d-1}\big(\partial \overline{B_s(x)} \cap \bigcup_{i = 1}^M \partial^* E_i\big) = 0$
for $\lambda$-a.e.\ $s \in (0,r_x)$, which we briefly argue by way of contradiction.
Assume that this assertion is false, then there exists a subset $A \subset (0,r_x)$
with positive Lebesgue measure, where $\Ha^{d-1}(\partial \overline{B_s} \cap \bigcup_{i = 1}^M \partial^* E_i) > \varepsilon_0$ for some $\varepsilon_0 > 0$ and a.e.\ $s \in A$.
This implies $\Ha^{d-1}(\overline{B_{\bar{s}}} \cap \bigcup_{i = 1}^M \partial^* E_i) = \infty$
for some $\bar{s} \le r_x$ by virtue of the Fubini--Tonelli theorem,
which contradicts that the $E_i$ are sets of finite perimeter in $\Omega$.
\end{proof}
Next, we briefly argue that when \eqref{eq:stationarity} is violated for some
$\phi \in C_c^\infty(\Omega,\R^d)$, then there exists a closed ball in the countable
set $\tilde{\mathcal{F}}$ from \cref{lem:cover} so that the restriction
of the integrand in \eqref{eq:stationarity} to this closed ball also implies a violation.
\begin{lemma}\label{lem:restriction_part_one}
Let $F : L^1(\Omega) \to \R$ and $\bar{w} \in \BVW(\Omega)$
satisfy the assumptions of \cref{dfn:patch_stationarity}.
Let $\calD \subset 2^{\Omega}$ satisfy \cref{ass:calD_open_cover}.
Let $\{E_1,\ldots,E_M\}$ be the Caccioppoli partition of
$\Omega$ associated with $\bar{w}$.
Let \eqref{eq:stationarity} be violated, that is
\begin{gather}\label{eq:instationarity}
\sum_{i=1}^{M - 1} \sum_{j=i + 1}^M
\int_{\partial^*{E}_i \cap \partial^* E_j}
\underbrace{(w_j - w_i)\nabla F(w)(x)\phi(x)\cdot n_{{E}_i}(x)
	- \alpha |w_i - w_j| \bdvg{E_i} \phi(x)}_{\eqqcolon \psi_{ij}(x)}\dd \Ha^{d-1}(x)
> \eta
\end{gather}
holds for some $\phi \in C_c^\infty(\Omega,\R^d)$ and $\eta > 0$.
Let $\tilde{\mathcal{F}}$ be as in \cref{lem:cover}.

Then, there exist $\overline{B_k} \in \tilde{\mathcal{F}}$
with $\overline{B_k} \subset D$ for some $D \in \calD$
and a positive scalar $\eta_D > 0$ such that
\begin{gather}\label{eq:restricted_instationarity}
\sum_{i=1}^{M - 1} \sum_{j=i + 1}^M
\int_{\partial^*{E}_i \cap \partial^* E_j}
\psi_{ij}(x) \chi_{\overline{B_k}}(x) \dd \Ha^{d-1}(x)
> \eta_D.
\end{gather}
\end{lemma}
\begin{proof}
	We first note that if the equality \eqref{eq:stationarity} is violated,
	then there is $\eta > 0$ such that the absolute value of the left hand side of
	\eqref{eq:instationarity} is greater than $\eta$. Thus by replacing $\phi$ with $-\phi$ if
	necessary, \eqref{eq:instationarity} holds.
	By virtue of the properties asserted in \cref{lem:cover} and the
	countable additivity of the measure
	$\Ha^{d-1}\mres \left(\bigcup_{i=1}^M \partial^* E_i\right)$,
	we obtain
	\[ \sum_{k\in\N} \sum_{i=1}^{M - 1} \sum_{j=i+1}^M
	\int_{\partial^*{E}_i \cap \partial^* E_j}
	\psi_{ij}(x) \chi_{\overline{B_k}}(x) \dd \Ha^{d-1}(x)
	> \eta > 0
	\]
	from the instationarity inequality \eqref{eq:instationarity}.
	Consequently, there has to exist a strictly positive summand
	on the left hand side. We use (one of the) patches $D$ asserted by
	\eqref{eq:fine_cover_patch_stationarity}.
\end{proof}
We now continue with the proof of
\cref{thm:stationarity_localization}.
\begin{proof}[Proof of \cref{thm:stationarity_localization}]
If $\bar{w} \in \BVW(\Omega)$ is stationary and $D \in \calD$,
then \eqref{eq:stationarity} holds for all
$\phi \in C_c^\infty(D,\R^d)$ because \cref{ass:calD_open_cover}
ensures that $D$ is an open subset of $\Omega$.

Thus it remains to show that if $\bar{w} \in \BVW(\Omega)$ is patch-stationary with respect to $\calD$,
then it is also stationary. We prove the claim by means of a contrapositive argument and assume
that $\bar{w}$ is not stationary. Let $\{E_1,\ldots,E_M\}$ be the
Caccioppoli partition of $\Omega$ associated with $\bar{w}$.
Because $\bar{w}$ is not stationary, there are $\phi \in C^\infty_c(\Omega,\R^d)$, $0 \neq \phi$,
and $\eta > 0$ such that the inequality \eqref{eq:instationarity}
holds. Let $\tilde{\mathcal{F}}$ be as in \cref{lem:cover}.
By virtue of \cref{lem:restriction_part_one} there exist
$\overline{B_k} \in \tilde{\mathcal{F}}$ and $D \in \calD$
such that $\overline{B_k} \subset D$ for some $D \in \calD$
and a positive scalar $\eta_D > 0$ such that
\eqref{eq:restricted_instationarity} holds.
We close the proof by constructing a $C_c^\infty(D,\R^d)$-function
that violates \eqref{eq:stationarity}.

To this end, we use a mollification of the restriction of $\phi$ onto
$\overline{B_k}$, that is we replace $\phi \chi_{\overline{B_k}}$
by a smooth function in $C^\infty_c(D,\R^d)$ in the integrand of \eqref{eq:restricted_instationarity}.
To this end, let $\phi_\delta \coloneqq \eta_\delta * (\chi_{\overline{B_k}} \phi)$ for a family of positive
standard mollifiers $(\eta_\delta)_{\delta > 0}$. Then $\phi_\delta \in C_c^\infty(\Omega,\R^d)$.
We show that there exists $\delta_0 > 0$ such that for all
$\delta \in (0,\delta_0)$ it holds that $\supp \phi_\delta \subset D$ and
\begin{gather}\label{eq:patch_instationarity}
\lim_{\delta \searrow 0}\sum_{i=1}^{M - 1} \sum_{j=i+1}^M
\int_{\partial^*{E}_i \cap \partial^* E_j}
\psi_{ij}^\delta(x)\dd \Ha^{d-1}(x) > \eta_D,
\end{gather}
where $\psi^\delta_{ij} \coloneqq w_i (-\nabla F(w))(\phi_\delta \cdot n_{{E}_i})
- \alpha |w_i - w_j| \bdvg{E_i} \phi_\delta$ for all $i$, $j \in \{1,\ldots, M\}$.

For all $y \in \interior \overline{B_k}$, we obtain
$\dvg \phi_\delta(y) \to \dvg \phi(y)$ and $\nabla \phi_\delta (y) \to \nabla \phi(y)$
as $\delta \searrow 0$. Moreover, $\overline{B_k} \in \tilde{\mathcal{F}}$ gives
\eqref{eq:no_Hdm1_contribution_of_ball_boundaries}
that $\Ha^{d-1}\big(\partial \overline{B_k} \cap \bigcup_{i = 1}^M \partial^* E_i\big) = 0$. In combination with the properties of
the mollification, this implies
\[ \dvg \phi_\delta(y) \to \dvg \phi(y)
\enskip\text{and}\enskip
\nabla \phi_\delta (y) \to \nabla \phi(y)
\quad\text{as }\delta \searrow 0
\quad\text{for $\Ha^{d-1}$-a.e.\ } y \in \bigcup_{i = 1}^M \partial^* E_i \cap \overline{B_k}
\]
and
\[ \dvg \phi_\delta(y) \to 0
\enskip\text{and}\enskip
\nabla \phi_\delta (y) \to 0
\quad\text{as }\delta \searrow 0
\quad\text{for all } y \in \Omega \setminus \overline{B_k}.
\]
Because $D$ is open, it holds that $\dist(\overline{B_k},\partial D) > 0$ and
we obtain that there exists $\delta_0 > 0$ such that for all $\delta \in (0, \delta_0)$
the inclusion $\supp \phi_\delta \subset \subset D$ is satisfied.
These considerations imply
\[ \psi_{ij}^\delta(y) \to \psi_{ij}(y)\chi_{\overline{B_k}}(y)
   \quad\text{as }\delta \searrow 0\quad
   \text{for $\Ha^{d-1}$-a.e.\ }y \in \overline{B_k} \cap \bigcup_{i = 1}^M \partial^* E_i.
\]

Because of the $L^\infty$-bounds on $\phi$ and $\chi_{\overline{B_k}}$, we
can apply Lebesgue's dominated convergence theorem and obtain
\begin{multline*}
\sum_{i=1}^{M - 1} \sum_{j=i+1}^M
\int_{\partial^*{E}_i \cap \partial^* E_j}
\psi_{ij}^\delta(x)\dd \Ha^{d-1}(x)
\underset{\delta \searrow 0}\to\\
\sum_{i=1}^{M - 1} \sum_{j=i+1}^M
\int_{\partial^*{E}_i \cap \partial^* E_j}
\psi_{ij}(x) \chi_{\overline{B_k}}(x)\dd \Ha^{d-1}(x)
> \eta_D > 0,
\end{multline*}
implying that the left hand side is eventually strictly greater than zero
so that $\phi_\delta$ eventually violates \eqref{eq:stationarity}
and satisfies $\supp \phi_\delta \subset D$.
\end{proof}

\section{Trust-region Patch Algorithm}\label{sec:algorithm}
We now introduce \cref{alg:slipsub_greedy}, which is a modification of Algorithm 5.1 in
\cite{manns2023on}, where we optimize subproblems over each patch in the
collection $\calD$ of patches determined {\it a priori}. Before stating it, we
require a regularity assumption on the main part of the objective $F$.
\begin{assumption}[see also Assumption 2.2 in \cite{manns2023homotopy}, Assumption 4.3 in \cite{manns2023on}]\label{ass:F_regularity}~
\begin{enumerate}
\item\label{itm:F_Frechet} Let $F : L^1(\Omega) \to \R$ be continuously Fr\'{e}chet differentiable.
\item\label{itm:nablaF_Lipschitz} Let $\nabla F : L^1(\Omega) \to L^\infty(\Omega)$ be Lipschitz continuous on the feasible set, 
that is
\[ \infty > L_{\nabla F} 
\coloneqq \sup\left\{ 
\frac{\|\nabla F(w_1) - \nabla F(w_2)\|_{L^\infty}}{\|w_1 - w_2\|_{L^1}}
\,\middle|\,
w_1(x), w_2(x) \in \conv W \text{ for a.e.\ } x \in \Omega\right\}.
\]
\end{enumerate}
\end{assumption}
We note that \cref{ass:F_regularity} \ref{itm:F_Frechet} already implies the boundedness
\begin{gather}\label{eq:cofb}
c(b) \coloneqq \sup\left\{ \|\nabla F(w)\|_{L^\infty} \,\middle|\, w(x) \in W \text{ a.e.\ and } \TV(w) < b \right\} < \infty.
\end{gather}
for every $b > 0$ because we can identify $L^1(\Omega)^*$ with $L^\infty(\Omega)$ and the bound $\TV(w) < b$ implies that
the supremum is taken over a relatively compact subset of $L^1(\Omega)$
due to the compact embedding $\BV(\Omega) \hookrightarrow L^1(\Omega)$ \cite[Corollary 3.49]{ambrosio2000functions}.

\algblockdefx[ParFor]{ParFor}{EndParFor}%
[1]{{\textbf{(parallel) for}}\ #1\ {\textbf{do}}}%
{{\textbf{end for}}}
\algblockdefx[VarFor]{VarFor}{EndVarFor}%
[2]{{\textbf{for}}\ #1\ {\textbf{do}}\ #2}%
{{\textbf{end for}}}
\begin{algorithm}[t]
	\caption{Sequential linear integer programming method
		with greedy patch updates}\label{alg:slipsub_greedy}

	\textbf{Input:} $F$ satisfying \cref{ass:F_regularity} with smoothness constant $L_{\nabla F}$,
	$\Delta^0 > 0$, $w^0 \in \BVW(\Omega)$, $\sigma \in (0,1)$, set of patches $\calD$.
	\begin{algorithmic}[1]
		\For{$n = 0,1,2,\ldots$}
		\State Set $\calA \gets \emptyset$, $\calW \gets \{ (0,D) \,|\, D \in \calD \}$.
		\For{$k = 0,1,\ldots$}\label{ln:tabulation_loop}
		\While{$(k,D) \in \calW$}\label{ln:inner_tabulation_loop}
		\State $\tilde{w}^{n,k,D} \gets$
		minimizer of $\text{{\ref{eq:trp}}}(w^{n}, \nabla F(w^{n}), D, \Delta^{0} 2^{-k})$. \label{ln:parallel_trstep}
		\State $\pred^{n,k,D} \gets
		(\nabla F(w^{n}), w^{n} - \tilde{w}^{n,k,D})_{L^2}
		+ \alpha \TV(w^{n}) - \alpha \TV(\tilde{w}^{k,n,D})$\label{ln:parallel_pred}
		\State $\ared^{n,k,D} \gets F(w^{n}) + \alpha \TV(w^{n})
		- F(\tilde{w}^{n,k,D}) - \alpha\TV(\tilde{w}^{n,k,D})$
		\If{$\ared^{n,k,D} \ge \sigma \pred^{n,k,D}$ \textbf{and} $\pred^{n,k,D} > 0$}\label{ln:suff_dec}
		\State $\calA \gets \calA\cup \{ (k,D) \}$.
		\ElsIf{$\pred^{n,k,D} > 0$ \textbf{and} $\max_{(\tilde{k},\tilde{D}) \in \calA} \ared^{n,\tilde{k},\tilde{D}}
			< \pred^{n,k,D} + L_{\nabla F} \Delta^{0}
			2^{-k}$}\label{ln:not_pred_zero_or_ared_domination}
		\State $\calW \gets \calW \cup \{(k + 1,D)\}$\label{ln:increase_k}
		\EndIf
		\State $\calW \gets \calW \setminus \{(k,D)\}$.
		\EndWhile
		\If{$\calW= \emptyset$}
		\State \textbf{break}
		\EndIf
		\EndFor
		\If{$\calA = \emptyset$}
		\State \textbf{return} ($w^n$ is stationary).
		\EndIf
		\State $\bar{w}^n \gets w^n$.
		\State $\bar{j}^0 \gets F(w^n) + \alpha \TV(w^n)$.
		\While{$\calA \neq \emptyset$}
		\State $\bar{k},\bar{D} \gets \argmax\{\ared^{n,k,D}\,|\,(k,D) \in \calA \}$\label{ln:greedy}
		\State $\tilde{w}^n \gets \bar{w}^n\chi_{\Omega\setminus \bar{D}}
		        + \chi_{\bar{D}}\tilde{w}^{n,\bar{k}, \bar{D}}$
		\State $\bar{j} \gets F(\tilde{w}^n) + \alpha \TV(\tilde{w}^n )$
	    \If{$\bar{j} < \bar{j}^0$}\label{ln:apply_reduction}
	    \State $\bar{w}^n \gets \tilde{w}^n$, $\bar{j}^0 \gets \bar{j}$,
	    $\calA \gets \calA\setminus \{ (\bar{k},\bar{D}) \}$.
		\Else
		\State \textbf{break}
		\EndIf
		\EndWhile
		\State $w^{n + 1} \gets \bar{w}^n$
		\EndFor
	\end{algorithmic}
\end{algorithm}
Specifically, for a given patch $D$, an iterate $\bar{w}$, a gradient (approximation) $g$ of $\nabla F(\bar{w})$,
and a trust-region radius $\Delta > 0$, the subproblem reads:
\begin{gather}\label{eq:trp}
\text{{\ref{eq:trp}}}(\bar{w}, g, D, \Delta) \coloneqq
\left\{
\begin{aligned}
\min_{w \in L^2(\Omega)}\ & (g, w - \bar{w})_{L^2(\Omega)} + \alpha \TV(w)-\alpha \TV(\bar{w})\\
\text{s.t.}\quad & \|w - \bar{w}\|_{L^1(\Omega)} \le \Delta
\text{ and }w(x) \in W \text{ for a.e.\ } x \in D,\\
& w(x) = \bar{w}(x) \text{ for a.e.\ } x \in \Omega\setminus D.
\end{aligned}
\right.
\tag{TRP}
\end{gather}
We briefly recall that $\text{{\ref{eq:trp}}}(\bar{w}, g, D, \Delta)$ is well-defined, that is, it admits a solution.
\begin{proposition}\label{prp:trp_existence}
Let $D \in \calD$, $\bar{w} \in \BVW(\Omega)$, $g \in L^2(\Omega)$, and $\Delta > 0$.
Then $\emph{\ref{eq:trp}}(\bar{w}, g, D, \Delta)$ has a minimizer.
\end{proposition}
\begin{proof}
A proof can be found in \cite[Proposition 3.2]{leyffer2022sequential}.
\end{proof}

The regularity condition \Cref{ass:F_regularity} as well as \Cref{prp:trp_existence} allow for
\Cref{alg:slipsub_greedy} to converge to a stationarity point defined in \eqref{eq:stationarity}.
The method itself consists of an outer loop indexed by $n$,
and an inner/block-update loop indexed by $k$.

At the beginning of each outer loop, we initialize two sets:
the set of acceptable step candidates, $\calA = \emptyset$; and
the working set of patches, $\calW = (0, D)$ for $D \in \calD$.
The inner/block-update loop starting at Line \ref{ln:tabulation_loop} solves a
subproblem \eqref{eq:trp} every iteration $k$ over all patches in the set $\calD$ for which
the pair $(k,D)$ is contained in $\calW$. These subproblem solves yield trial iterates $\tilde{w}^{n,k,D}$, with
corresponding predicted, $\pred^{n,k,D}$, and actual, $\ared^{n,k,D}$, reductions.
Based on their values, $\calW$ is updated by means of the following case distinction.
If $\pred^{n,k,D}$ is zero, this means that $w^n$ is stationary in the patch $D$
and  $(k,D)$ is just removed from $\calW$. If this is not the case and the block update
$\tilde{w}^{n,k,D}$ moreover sufficiently decreases the cost function by some fraction of
the predicted reduction greater than zero, it is added to $\calA$ as an ``acceptable step''
and $(k,D)$ is also removed from $\calW$. If $\pred^{n,k,D}$ is greater than zero but
the sufficient decrease is not satisfied, it is checked if
$\max_{(\tilde{k},\tilde{D}) \in \calA} \ared^{n,\tilde{k},\tilde{D}}$,
the maximum of already found acceptable reductions, is dominated
from above by $\pred^{n,k,D} + L_{\nabla F} \Delta^{0}$,
the predicted reduction plus the Lipschitz constant of the first part of the objective
times the trust-region radius. We highlight that we use the convention
$\max \emptyset = -\infty$ here.
If this condition holds, a further reduction of the trust-region
radius may give an acceptable step that then maximizes the predicted reductions over all
steps that have been deemed acceptable so far. Thus $(k,D)$ is replaced by $(k+1,D)$ in $\calW$
so that the same patch but with reduced trust-region radius is now contained in the working set
$\calW$. If this condition does not hold, this implies that even if there is an acceptable
step for this patch with a smaller trust-region radius, the resulting actual reduction cannot
give the maximizer of $\max_{(\tilde{k},\tilde{D}) \in \calA} \ared^{n,\tilde{k},\tilde{D}}$
and, consequently, $(k,D)$ is removed from $\calW$ without replacement.
This condition is derived from \Cref{ass:F_regularity} and allows us to prove finite termination of the 
inner/block-update loop if $w^n$ is not stationary because it can happen $w^n$ is stationary on $D$
but has positive predicted reduction for all positive trust-region radii, which would lead to infinite reduction
of the trust-region radius without this safeguarding.
Once $\calW = \emptyset$, that is, if there are no more working patches we terminate
\Cref{alg:slipsub_greedy}.
An example run of this inner/block-update loop for four patches is given in \Cref{tbl:ex_tab_loop}.
We note that a straightforward parallel version of 
\cref{alg:slipsub_greedy} arises by optimizing in parallel over the elements
of $\calW$. However, our computational results in \cref{sec:2Dtestcase} indicate
that the runtimes between the elements of $\calW$ may vary dramatically and we
find it more advisable to use parallelization on the subproblem solver level.

\begin{table}
\begin{center}
  \begin{tabular}{ c | c c c c | c | c }
\toprule
   \diaghead{\theadfont $\calD$}%
  {$\calD$}{$k$ }&
\thead{$D_1$}&\thead{$D_2$} & \thead{$D_3$} & \thead{$D_4$} & $\calA$ & $\calW$\\
   \midrule
   0 & \darkgreen{$\tilde{w}^{n,0, 1}$} & \darkgreen{$\tilde{w}^{n,0, 2}$} & \darkorange{$\tilde{w}^{n,0, 3}$} & \darkgreen{$\tilde{w}^{n,0, 4}$} & $\emptyset$ & $\{(1,1), (1, 2), (1, 4)\}$ \\
   1 & \darkgreen{$\tilde{w}^{n,1, 1}$} & \darkgreen{$\tilde{w}^{n,1, 2}$} & -- & \blue{$\tilde{w}^{n,1, 4}$} & $\{(1,4)\}$ & $\{(2,1), (2, 2)\}$  \\
   2 & \darkgreen{$\tilde{w}^{n,2, 1}$} & \violet{$\tilde{w}^{n,2, 2}$} & -- & --  & $\{(1,4)\}$ & $\{(3, 1)\}$  \\
   3 & \darkgreen{$\tilde{w}^{n,3, 1}$} & -- & -- & -- & $\{(1,4)\}$ & $\{(4,1)\}$  \\
   4 & \blue{$\tilde{w}^{n,4, 1}$}  & -- & -- &--  & $\{(4, 1),(1,4)\}$ & $\emptyset$ \\
   \bottomrule
  \end{tabular}
  \end{center}
  \caption{Example inner-loop iteration for $\vert \calD\vert = 4$ at line 17.
  	\darkorange{Orange} signifies that $\pred^{n,k,D} = 0$ so that $w^n$ is already stationary on
  	$D$ and therefore $(k,D)$ can safely be removed from $\calW$.
  	\darkgreen{Green} signifies that signifies Line \ref{ln:not_pred_zero_or_ared_domination} holds:
  	$\tilde{w}^{n,k,D}$ is not yet acceptable and also not dominated by some previous iterate
  	that was found acceptable. \blue{Blue} signifies Line \ref{ln:suff_dec} holds: $\tilde{w}^{n,k,D}$ is
  	acceptable and $(k,D)$ can be transferred from $\calW$ to $\calA$.
  	\violet{Violet} signifies that $\pred^{n,k,D} > 0$ but neither Line \ref{ln:not_pred_zero_or_ared_domination}
  	nor Line \ref{ln:suff_dec} holds: the possible actual reduction for all further trust-region radii
  	is dominated by some element of $\calA$ and $(k,D)$ can safely be removed from $\calW$ too.}
  \label{tbl:ex_tab_loop}
\end{table}

In the outer loop, once $\calA$ is the nullset, then by construction
$w^n$ is stationary and we terminate the algorithm.
If $\calA$ is nonempty, then we define trial variable
$\bar{w}^n = w^n$ and compute trial function value $\bar{j}^0$.
We then enter a ``greedy'' update loop; this loop inserts the
patch updates to the overall solution and measure decrease
of the actual function. In effect, we iterate over stored
maximum block actual reductions
$(\bar k, \bar D)\gets \argmax\{\ared^{n,k,D}|(k,D)\in \calA\}$, insert the resulting solution
$\tilde{w}^n$ associated with $(\bar k, \bar D)$ into $\bar{w}^n$ and compute a new trial function value.
If this function value is less than $\bar{j}^0$, then we eliminate this $(k, D)$ from
the set of acceptable steps, keep the $\bar{w}^n$ with the associated $\tilde{w}^n$, and repeat
until $\calA= \emptyset$.

By construction, the first (greedy) patch update will always satisfy the acceptance criterion;
further updates are heuristic improvements that are
accepted until the cost function increases.

\section{Convergence Analysis of \texorpdfstring{\cref{alg:slipsub_greedy}}{TEXT}}\label{sec:convergence}
We provide a convergence analysis for a greedy patch selection approach
under two assumptions on the set of patches.
\begin{assumption}[Sufficient overlap and regularity of patches]\label{ass:patches}
Let $\Omega$ be a bounded Lipschitz domain.
Let the finite set of patches $\calD \subset 2^\Omega$ satisfy the following conditions.
\begin{enumerate}[label=(\alph*)]
\item For all $x \in \Omega$, there exist $r > 0$ and $D \in \calD$ such that
      $\overline{B_r(x)} \subset \subset D$. (Patch overlap)\label{itm:overlap}
\item For all $D \in \calD$, we assume that $D \in \calD$ is a bounded Lipschitz domain.
     (Patch regularity)\label{itm:regularity}
\end{enumerate}
\end{assumption}
Because $\calD$ is finite, \Cref{ass:patches} \ref{itm:overlap} is equivalent
to \cref{ass:calD_open_cover} and in particular, the analysis in \cref{sec:localization} may be applied.
Consequently, \cref{thm:stationarity_localization} gives that the instationarity condition
\eqref{eq:instationarity} implies a localization to at least one patch,
on which patch-instationarity \eqref{eq:patch_instationarity} holds.
\Cref{ass:patches} \ref{itm:regularity} implies that the patches themselves are sets of finite
perimeter and the assumptions for the analysis of the inner loop of Algorithm 1 from \cite{manns2023on}
can be applied to trust-region subproblems $\text{{\ref{eq:trp}}}(w^{n-1}, \nabla F(w^{n-1}), D^n, \Delta^{0}2^{-k})$
for some $\Delta^0 > 0$ as are generated by \cref{alg:slipsub_greedy}.
While the analysis is not impaired by \cref{ass:patches} \ref{itm:regularity},
it facilitates the intuition greatly when considering patches that are convex and
have boundaries that are finite unions of $d-1$-dimensional convex polyhedra.

In order to analyze \cref{alg:slipsub_greedy} and the trust-region subproblems,
we briefly recall what we mean by $\ared$ and $\pred$ for a
given sequence $\{w^{\ell}\}_\ell \subset \BVW(\Omega)$ that converges strictly to some
limit $w \in \BVW(\Omega)$.
Let $k \in \N$, $D \in \calD$, and $\tilde{w}^{\ell,k,D}$ be a solution to
$\text{{\ref{eq:trp}}}(w^{\ell}, \nabla F(w^{\ell}), D, \Delta^{0} 2^{-k})$,
which exists by means of \cref{prp:trp_existence}. Then we define
\begin{align*}
{\pred}^{\ell,k,D} &\coloneqq (\nabla F(w^{\ell}), w^{\ell} - \tilde{w}^{\ell,k,D})_{L^2(\Omega)}
+ \alpha \TV(w^{\ell}) - \alpha \TV(\tilde{w}^{\ell,n,D}) \text{ and}\\
{\ared}^{\ell,k,D} &\coloneqq F(w^{\ell}) + \alpha \TV(w^{\ell})
- F(\tilde{w}^{n,k,D}) - \alpha\TV(\tilde{w}^{\ell,k,D}).
\end{align*}
We now aim to show that if the limit of such a sequence is instationary, 
then there is a patch such that the sufficient decrease condition that is also in 
Line \ref{ln:suff_dec} is eventually satisfied.
Our argument uses so-called competitors that modify a weakly-$^*$ converging sequence $w^\ell \weakstarto w$ in 
$\BVW(\Omega)$ so that the resulting sequence coincides with $\{w^\ell\}_\ell$ on $D^c$
but has properties of a 
different function $\hat{w}$ inside a patch $D$ and
does not affect the boundaries of the level sets close to $\partial D$.
Their existence is asserted below.
\begin{lemma}\label{lem:competitor}
Let \cref{ass:patches} hold. Let $D \in \calD$ and $D_r \coloneqq \{ x \in D \,|\, \dist(x, D^c) > r \}$ for $r > 0$.
Let $w^\ell \weakstarto w$ in $\BVW(\Omega)$. Let $\hat{w} \in \BVW(\Omega)$ and $K \subset \subset D$ with
$\supp w - \hat{w} \subset \subset K$.

Then there exists $s > 0$ only depending on $K$ such that
$w$, $w^\ell$ for $\ell \in \N$, and
the functions defined by
\begin{gather}\label{eq:construct_wkl}
\hat{w}^{\ell}(x) \coloneqq \left\{
\begin{aligned}
w^{\ell}(x) & \text{ if } x \notin D_{s}, \\
\hat{w}(x) & \text{ else}
\end{aligned}
\right.
\end{gather}
for $\ell \in \N$ satisfy $\hat{w}^{\ell} \weakstarto \hat{w}$ in $\BVW(\Omega)$ as $\ell \to \infty$ and
\begin{align*}
\TV(w)          &= \TV{}_{\Omega\setminus \overline{D_s}}(w) + \TV{}_{D_s}(w) \\
\TV(\hat{w}^{\ell}) &= {\TV}_{\Omega\setminus \overline{D_s}}(w^{\ell}) + {\TV}_{D_s}(\hat{w}) + b^\ell
\end{align*}
for some $\{b^\ell\}_\ell \subset [0,\infty)$ with $\liminf_{\ell \to \infty} b^\ell = 0$. Moreover,
if $w^\ell \to w$ also strictly in $\BVW(\Omega)$, we have 
\begin{align*}
{\TV}_{\Omega\setminus \overline{D_s}}(w^\ell) \to {\TV}_{\Omega\setminus \overline{D_{s}}}(w)&\text{ and}\\
\quad{\TV}_{D_s}(w^\ell) \to {\TV}_{D_s}(w)&.
\end{align*}
\end{lemma}
\begin{proof}
Let $\{E_1,\ldots,E_M\}$, $\{E_1^\ell,\ldots,E_M^\ell\}$, $\{\hat{E}_1,\ldots,\hat{E}_M\}$, and
$\{\hat{E}_1^{\ell},\ldots,\hat{E}_M^{\ell}\}$ denote the Caccioppoli partitions of $\Omega$ such that
the identities $w = \sum_{i=1}^M w_i \chi_{E_i}$, $w^\ell = \sum_{i=1}^M w_i \chi_{E_i^\ell}$, $\ell \in \N$,
$\hat{w} = \sum_{i=1}^M w_i \chi_{\hat{E}_i}$, and $\hat{w}^{\ell} = \sum_{i=1}^M w_i \chi_{\hat{E}_i^{\ell}}$,
$\ell \in \N$, hold \cite[Lemma 2.1]{manns2023on}. Then we observe that there is a small enough $r_0 > 0$ such that
$K \subset \subset D_{r_0} \subset \subset D$.

Next, we argue similar to \cite[Theorem 21.14]{maggi2012sets} to identify $0 < s < r_0$ such that
modifying $w^\ell$ on $D_s$ to obtain the competitor $\hat{w}^{\ell}$ implies the claimed properties. Specifically,
we need that i) $\partial D_s$ does not intersect with the reduced boundaries $\partial^*E_i$ and $\partial^*E_i^\ell$ 
 (for $\ell \in \N$) on a set of strictly positive $\Ha^{d-1}$-measure, and ii) the $\Ha^{d-1}$-measure of the intersection of 
$\partial D_s$ and the symmetric difference of the points of density one of $E_i$ and $E_i^\ell$ vanishes
as $\ell \to \infty$. This is used to ensure that $\partial D_s$ does not affect the boundary of the level sets of the 
competitor and thus the value of its total variation.

To this end, let $r_{\max}$ be large enough such that $D_{r_{\max}} = \emptyset$. We obtain for all $r > 0$ that
\[
\|w - w^\ell\|_{L^1}
\ge \sum_{i=1}^M\big|D_r\cap \big(E_i^\ell \symdif E_i\big)\big|
= \int_r^{r_{\max}} \sum_{i=1}^M \Ha^{d-1}\left(\partial D_s \cap
\left((E_i^\ell)^{(1)} \symdif E_i^{(1)}\right)\right) \dd s,
\]
where the second identity holds due to the coarea formula, which can be applied using the
Lipschitz continuity of the function $x \mapsto \dist(x, D^c)$.

From Proposition 2.16 in \cite{maggi2012sets}, we obtain for all $i \in \{1,\ldots,M\}$ and
a.e.\ $s > 0$ that
\begin{align}
\Ha^{d-1}(\partial D_{s} \cap \partial^* E_i) &= 0 \text{ and} \label{eq:D_bnd_zero_intersection_limit}\\
\Ha^{d-1}(\partial D_{s} \cap \partial^* E_i^\ell) &= 0
\text{ for all } \ell \in \N.\label{eq:D_bnd_zero_intersection_iterate}
\end{align}
We apply Fatou's lemma and obtain
\begin{gather}\label{eq:D_liminf_to_zero}
\liminf_{\ell \to \infty}
\sum_{i=1}^M \Ha^{d-1}\left(\partial D_s \cap
\left((E_i^\ell)^{(1)} \symdif E_i^{(1)}\right)\right)
= 0
\end{gather}
for a.e.\ $s > 0$.

Let $s \in (0,r_0)$ satisfy	\eqref{eq:D_bnd_zero_intersection_limit}, \eqref{eq:D_bnd_zero_intersection_iterate}
for all $\ell \in \N$, and \eqref{eq:D_liminf_to_zero}.
Then $\hat{w}^{\ell} \weakstarto \hat{w}$ in $\BVW(\Omega)$ as $\ell \to \infty$ as claimed.

Next, we use \cref{lem:eq_1635} from the literature to characterize the distributional derivative of $\chi_{\hat{E}_i^\ell}$.
To verify its assumptions, we first recall that $\Ha^{d-1}(\partial D_s\cap\partial^*E_i^\ell) = 0$ holds for all $\ell \in \N$
by virtue of \eqref{eq:D_bnd_zero_intersection_iterate}. Second, $\hat{w}(x) \neq w(x)$ can only hold for
$x \in K \subset \subset D_{s}$ so that
$\Ha^{d-1}(\partial D_s\cap\partial^*\hat{E}_i) = \Ha^{d-1}(\partial^*D_s\cap\partial^*E_i) = 0$
holds by virtue of \eqref{eq:D_bnd_zero_intersection_limit}. Finally, we observe that
$\hat{E}_i^{\ell} = (\hat{E}_i \cap D_s) \cup (E_i^\ell \setminus D_s)$ holds by definition in
\eqref{eq:construct_wkl}. We thus apply \cref{lem:eq_1635} and obtain
\begin{align*}
D\chi_{\hat{E}_i^{\ell}}
&= D\chi_{E_i^{\ell}} \mres (\Omega \setminus \overline{D_{s}})
+ D\chi_{\hat{E}_i} \mres D_{s} \\
&\hphantom{=} + D\chi_{D_{s}} \mres \left((\hat{E}_i)^{(1)} \cap (E_i^{\ell})^{(0)}\right)
- D\chi_{D_{s}} \mres \left((\hat{E}_i)^{(0)} \cap (E_i^{\ell})^{(1)}\right).
\end{align*}
We employ the $\sigma$-additivity of $\Ha^{d-1}$ and the identity
$(E_i)^{(b)} \cap \partial D_{s} = (\hat{E}_i)^{(b)} \cap \partial D_{s}$
for $b \in \{0,1\}$ due to $K \subset \subset D_s$
to obtain
\begin{align*}
\Ha^{d-1} \mres \partial^*\hat{E}_i^{\ell}
&= \Ha^{d-1} \mres \left(\partial^*E_i^{\ell} \cap (\Omega \setminus \overline{D_{s}})\right)\\
&\hphantom{=}+  \Ha^{d-1} \mres \left(\partial^*\hat{E}_i \cap D_{s}\right)\\
&\hphantom{=}+  \Ha^{d-1} \mres \left(\partial D_{s} \cap (E_i)^{(1)} \cap (E_i^{\ell})^{(0)}\right)\\
&\hphantom{=}+  \Ha^{d-1} \mres \left(\partial D_{s} \cap (E_i)^{(0)} \cap (E_i^{\ell})^{(1)}\right)
\end{align*}
for the corresponding variation measures.
We insert $\partial^* \hat{E}_j^{\ell}$, $j \neq i$, on both sides, use
$\hat{E}_i^{\ell} \cap D_s = \hat{E}_i \cap D_s$ and
$\hat{E}_i^{\ell} \cap (\Omega \setminus \overline{D_{s}})
= E_i^{\ell} \cap (\Omega \setminus \overline{D_{s}})$
due to \eqref{eq:construct_wkl},
and the disjointness of the sets in the last two terms to obtain
\begin{align*}
\Ha^{d-1}\big(\partial^*\hat{E}_i^{\ell} \cap \partial^* \hat{E}_j^{\ell}\big)
&= \Ha^{d-1}\left(\partial^*E_i^{\ell} \cap \partial^* E_j^{\ell}
\cap (\Omega \setminus \overline{D_s})\right) \\
&\hphantom{=}+
\Ha^{d-1}\left(\partial^*\hat{E}_i \cap \partial^* \hat{E}_j \cap D_s\right)\\
&\hphantom{\le}+  \Ha^{d-1} \left(
\partial^* \hat{E}_j^{\ell} \cap
\partial D_s \cap \Big(\big((E_i)^{(1)} \cap
(E_i^{\ell})^{(0)} \big) \cup
\big((E_i)^{(0)} \cap (E_i^{\ell})^{(1)}\big)\Big)\right).
\end{align*}
Using \cref{lem:symdif} on the symmetric difference of sets of points of density $1$, we obtain
\begin{align*}
\Ha^{d-1}\big(\partial^*\hat{E}_i^{\ell} \cap \partial^*\hat{E}_j^{\ell} \cap \Omega \big)
&= \Ha^{d-1}\big(\partial^*E_i^{\ell} \cap \partial^* E_j^{\ell} \cap (\Omega \setminus \overline{D_s})\big)
\\
&\hphantom{=}+
\Ha^{d-1}\big(\partial^* \hat{E}_i \cap \partial^* \hat{E}_j \cap D_s\big)\\
&\hphantom{=}+  \underbrace{\Ha^{d-1} \left(
	\partial^*\hat{E}_j^{\ell} \cap \partial D_s
	\cap \big((E_i)^{(1)} \symdif (E_i^{\ell})^{(1)}\big)\right)}_{\eqqcolon b_{ij}^\ell}.
\end{align*}
Multiplying by $|w_i - w_j|$ and summing over $i$ and $j$ yields
\[ \TV(\hat{w}^{\ell}) = {\TV}_{\Omega\setminus \overline{D_s}}(w^{\ell}) + {\TV}_{D_s}(\hat{w}) + b^\ell. \]
with $b^\ell \coloneqq \sum_{i=1}^{M-1}\sum_{j=i+1}^M|w_i - w_j|b_{ij}^\ell$,
where \eqref{eq:D_liminf_to_zero} gives $\liminf_{\ell \to \infty} b^\ell = 0$.
Moreover, from \eqref{eq:D_bnd_zero_intersection_limit}
and $K \subset \subset D_{s}$, we obtain
\[ \TV(w) = {\TV}_{\Omega\setminus \overline{D_s}}(w) + {\TV}_{D_s}(w). \]
In addition, \eqref{eq:D_bnd_zero_intersection_limit}, $\TV(w^\ell) \to \TV(w)$,
and the lower semicontinuity of the total variation yield
\[
{\TV}_{\Omega\setminus \overline{D_s}}(w^\ell)
\to {\TV}_{\Omega\setminus \overline{D_{s}}}(w)
\quad\text{ and }\quad
{\TV}_{D_s}(w^\ell) \to {\TV}_{D_s}(w). \]
\end{proof}
\begin{lemma}\label{lem:instationarity}
	Let \cref{ass:F_regularity} hold. Let $w \in \BVW(\Omega)$ not be stationary. Let $\{w^\ell\}_\ell \subset \BVW(\Omega)$ converge 
	strictly to $w$. Let $\nabla F(w) \in C(\bar{\Omega})$.
	Then there exist $D \in \calD$, $k_0 \in \N$, $\ell_0 \in \N$, and $\varepsilon > 0$ such that
	there are infinitely many $\ell \ge \ell_0$ that satisfy
	\[ {\ared}^{\ell,k_0,D} \ge \sigma {\pred}^{\ell,k_0,D}
	   \enskip\text{and}\enskip
	   {\pred}^{\ell,k_0,D} > \varepsilon. \]
\end{lemma}
\begin{proof}
	Because $w$ is not stationary, \cref{thm:stationarity_localization} gives
	that there exists a patch $D \in \calD$ such that \eqref{eq:stationarity} is violated on $D$,
	that is there exist $\eta > 0$ and $\phi \in C_c^\infty(\Omega,\R^d)$ with $\supp \phi \subset \subset D$ such that
	\begin{gather}\label{eq:instationarity_on_D}
	\sum_{i=1}^{M - 1} \sum_{j=i + 1}^M
	\int_{\partial^*{E}_i \cap \partial^* E_j}
	(w_j - w_i)\nabla F(w)(x)\phi(x)\cdot n_{E_i}(x)
	- \alpha |w_i - w_j| \bdvg{E_i} \phi(x)
	\dd \Ha^{d-1}(x) > \eta.
	\end{gather}
	Let $(f_t)_{t \in (-\varepsilon, \varepsilon)}$ be defined by $f_t(x) \coloneqq x + t \phi(x)$
	for $x \in \Omega$ and $t \in (-\varepsilon,\varepsilon)$ and
	$f_t^{\#} w \coloneqq \sum_{i=1}^M w_i \chi_{f_t(E_i)}$ if $\{E_1,\ldots,E_M\}$ is
	the Caccioppoli partition of $\Omega$ such that
	$w = \sum_{i=1}^M w_i \chi_{E_i}$.

	Lemma 3.8 in \cite{manns2023on} gives that we can choose a sequence
	$t^k \searrow 0$ and $k_0 \in \N$ such that for all $k \ge k_0$ we have
	$f_{t^k} \in C_c^\infty(D,\R^d)$ and
	\begin{gather}\label{eq:L1_locvar}
	\|f_{t^k}^{\#} w - w\|_{L^1} = \left(1 - \frac{1}{k}\right) \Delta^0 2^{-k}.
	\end{gather}
	Moreover, Lemmas 3.3 and 3.5 in \cite{manns2023on} give that there exists
	some function $g : (-\varepsilon,\varepsilon) \to \R$ that satisfies
	\begin{gather}\label{eq:pred_locvar}
	g(t) \in o(t)\quad\text{ and }\quad
	-(\nabla F(w), f_{t^k}^{\#} w - w)_{L^2}
	- (\alpha \TV(f_{t^k}^{\#} w) - \alpha \TV(w))
	\ge t^k \eta + g(t^k).
	\end{gather}
	Unfortunately, this violation cannot be used directly to obtain a bound from
	below on $\pred^{\ell,k,D}$ because $f_{t^k}^{\#} w$ is infeasible
	for $\textrm{{\ref{eq:trp}}}(w^\ell, \nabla F(w^\ell), D, \Delta^0 2^{-k_0})$
	since it does not agree with $w^\ell$ on $\Omega \setminus D$. Therefore,
	in order to relate this reduction of the linearized objective back to
	$\pred^{\ell,k,D}$, we construct competitors $w^{\ell,k}$, which eventually become
	feasible for $\textrm{{\ref{eq:trp}}}(w^\ell, \nabla F(w^\ell), D, \Delta^0 2^{-k_0})$
	for all	$\ell \ge \ell_0$. The necessary interdependent choices for $k_0 \in \N$
	and $\ell_0 \in \N$ will be determined after providing the competitor construction.

	\textbf{Competitor construction.}
	To this end, we use that the local variation $(f_t)_{t \in (-\varepsilon,\varepsilon)}$
	induced by $\phi$ has $\supp f_t - I \subset \subset D$ for all $t \in (-\varepsilon, \varepsilon)$
	because of $\supp \phi \subset \subset D$. Specifically, we apply \cref{lem:competitor} with the choices
	$\hat{w} = f_{t^k}w$, $k \in \N$, and $K = \supp \phi$ and obtain that there exists
	$D_s \subset \subset D$ with $\supp \phi \subset \subset D_s$ such that for all
	$k \in \N$ there are functions $\{w^{\ell,k}\}_\ell$ such that $w^{\ell,k}(x) = w^\ell(x)$ for
	$x \in \Omega\setminus D_s$ and $w^{\ell,k}(x) = f_{t^k}^{\#}w(x)$ for $x \in D_s$ such that
	\begin{align*}
	w^{\ell,k} &\weakstarto f^{\#}_{t^k}w \text{ in } \BVW(\Omega) \text{ for } \ell \to \infty,\\
	\TV(w^{\ell,k}) &= {\TV}_{\Omega\setminus \overline{D_s}}(w^{\ell})
	+ {\TV}_{D_s}(f_{t^k}^{\#}w) + b^\ell\text{ with } b^\ell \ge 0 \text{ and }\liminf_{\ell\to\infty} b^\ell = 0,\\
	\TV(w) &= {\TV}_{\Omega\setminus \overline{D_s}}(w) + {\TV}_{D_s}(w),\\
	{\TV}_{\Omega\setminus \overline{D_s}}(w^\ell)
	&\to {\TV}_{\Omega\setminus \overline{D_{s}}}(w),\\
	{\TV}_{D_s}(w^\ell) &\to {\TV}_{D_s}(w).
	\end{align*}

	\textbf{Determination of $k_0 \in \N$ and $\ell_0 \in \N$.} Let $k_0 \in \N$ be such that
	\eqref{eq:L1_locvar} and \eqref{eq:pred_locvar}
	hold for all $k \ge k_0$. Then $w^\ell \to w$ strictly, the competitor properties above,
	and  \cref{ass:F_regularity} \ref{itm:F_Frechet} imply that we can find $\ell_0(k_0) \in\N$ such that
	\begin{gather}\label{eq:bounds_to_zero}
	\begin{aligned}
	\|w^\ell - w\|_{L^1({D_s})} &\le \frac{1}{k_0} \Delta^0 2^{-k_0}, \\
	\|\nabla F(w^\ell) - \nabla F(w)\|_{L^\infty({D_s})} &\le \frac{1}{k_0}, \\
	|{\TV}_{D_s}(w^\ell) - {\TV}_{D_s}(w)| &\le \frac{1}{k_0} \Delta^0 2^{-k_0},\text{ and}\\
	b^\ell &\le \frac{1}{k_0} \Delta^0 2^{-k_0}
	\end{aligned}
	\end{gather}
	hold for infinitely many $\ell \ge \ell_0(k_0)$.
	Then the triangle equality gives that these $w^{\ell,k_0}$ are feasible for
	$\textrm{{\ref{eq:trp}}}(w^{\ell}, \nabla F(w^{\ell}), D, \Delta^0 2^{-k_0})$.
	For these $\ell$, we obtain from \cref{ass:F_regularity} \ref{itm:F_Frechet} that
	\[ {\ared}^{\ell,k_0,D} = {\pred}^{\ell,k_0,D} + o(\Delta^0 2^{-k_0}) \]
	Consequently, the feasibility of $w^{\ell,k_0}$ and thus suboptimality for
	$\textrm{{\ref{eq:trp}}}(w^\ell, \nabla F(w^\ell), D, \Delta^0 2^{-k_0})$ gives for
	\[ 
	p^{\ell,k_0} \coloneqq (\nabla F(w^\ell),w^\ell - w^{\ell,k_0})_{L^2}
	+ \alpha \TV(w^\ell) - \alpha \TV(w^{\ell,k_0}) \]
	that ${\pred}^{\ell,k_0,D} \ge p^{\ell,k_0}$ and in turn
	\begin{align*}
	{\ared}^{\ell,k_0,D}
	\ge \sigma {\pred}^{\ell,k_0,D} + (1-\sigma)p^{\ell,k_0}
	+ o(\Delta^0 2^{-k_0})
	\end{align*}
	hold.

	We consider the first term in $p^{\ell,k_0}$.
	After inserting suitable zeros, in particular using $w^\ell = w^{\ell,k_0}$
	in $\Omega\setminus \overline{D_s}$ and $\supp \phi \subset D_s$,
	we obtain by means of \eqref{eq:bounds_to_zero} that
	\begin{align*}
	(\nabla F(w^\ell),w^\ell - w^{\ell,k_0})_{L^2(\Omega)} &= (\nabla F(w), w -
	w^{\ell,k_0})_{L^2(D_s)}
	+ (\nabla F(w^\ell) - \nabla F(w), w - w^{\ell,k_0})_{L^2(D_s)}  \\
	&\quad + (\nabla F(w^\ell), w^\ell - w)_{L^2(D_s)} \\
	&\ge (\nabla F(w), w - f_{t^{k_0}}^{\#}w)_{L^2(\Omega)} \\
	&\quad -\|\nabla F(w^\ell) - \nabla F(w)\|_{L^\infty(D_s)}
	        \|w - f_{t^{k_0}}^{\#}w\|_{L^1(D_s)} \\
	&\quad - \|\nabla F(w^\ell)\|_{L^\infty(D_s)}\|w^\ell - w\|_{L^1(D_s)}\\
	&\ge (\nabla F(w), w - f_{t^{k_0}}^{\#}w)_{L^2(\Omega)} \\
	&\quad - \frac{1}{k_0}\|f_{t^{k_0}}^{\#}w - w\|_{L^1(D_s)} - c \|w^\ell - w\|_{L^1(D_s)}\\
	&= (\nabla F(w), w - f_{t^{k_0}}^{\#}w)_{L^2(\Omega)} 
	\underbrace{-\left(\frac{1}{1 - k_0} + c\right)\frac{1}{k_0} \Delta^0 2^{-k_0}}_{= o(\Delta^02^{-k_0})}.
	\end{align*}
	with $c = \sup_{\tilde{\ell} \ge \ell_0(k_0))} \|\nabla F(w^{\tilde{\ell}})\|_{L^\infty(D_s)}$, which is bounded because of the continuous differentiability of $F : L^1(\Omega) \to \R$ and the set $\{w^\ell\,|\,\ell \in \N\} \cup \{w\}$
	being a compact in $L^1(\Omega)$.
	
	Next, we consider the difference $\TV(w^\ell) - \TV(w^{\ell,k_0})$ and obtain
	with a similar procedure
	\begin{align*}
	\TV(w^\ell) - \TV(w^{\ell,k_0}) &= {\TV}_{D_s}(w) - {\TV}_{D_s}(w^{\ell,k_0})
	- b^\ell\\
	&\quad + {\TV}_{D_s}(w^\ell) - {\TV}_{D_s}(w) \\
	&\ge {\TV}(w) - {\TV}(f_{t^{k_0}}^{\#}w) \underbrace{- \frac{2}{k_0} \Delta^0 2^{-k_0}}_{=o(\Delta^02^{-k_0})}
	\end{align*}
	for all of the infinitely many suitable $\ell \ge \ell_0(k_0)$.

	In combination, we can deduce
	\begin{gather}\label{eq:pk_lower_bound}
	\begin{aligned}
	p^{\ell,k_0}
	&\ge (\nabla F(w),w - f_{t^{k_0}}^{\#}w)_{L^2}
	+ \alpha \TV(w^\ell) - \alpha \TV(f_{t^{k_0}}^{\#}w) + o(\Delta^02^{-k_0})\\
	&\ge t^{k_0} \eta + o(t^{k_0}) + o(\Delta^02^{-k_0}),
	\end{aligned}
	\end{gather}
	where the second inequality follows from \eqref{eq:pred_locvar}.
	Using \eqref{eq:L1_locvar}, we obtain
	\[ \Delta^02^{-k_0} = \frac{k_0}{k_0 - 1}\big\|f_{t^{k_0}}^{\#}w - w\big\|_{L^1}
	\le \kappa t^{k_0}
	\]
	for some $\kappa > 0$, where the last inequality follows from Lemma 3.8 in \cite{manns2023on}.
	Consequently, we obtain
	\[ {\ared}^{\ell,k_0,D}
	\ge \sigma {\pred}^{\ell,k_0,D} + (1-\sigma) \kappa t^{k_0} + o(t^{k_0}).
	\]
	Therefore, we can choose $k_0 \in \N$ large enough such that the sum of the second and third term
	is positive. We now set $\ell_0 \coloneqq \ell_0(k_0)$ and obtain for infinitely
	many $\ell \ge \ell_0$ due to \eqref{eq:D_liminf_to_zero} that
	${\ared}^{\ell,k_0,D} \ge \sigma {\pred}^{\ell,k_0,D}$
	By potentially choosing $k_0 \in \N$ larger and adjusting $\ell_0$ accordingly,
	we also obtain ${\pred}^{\ell,k_0,D} \ge p^{\ell,k_0} > \varepsilon > 0$
	for some $\varepsilon > 0$ from \eqref{eq:pk_lower_bound}.
\end{proof}

\begin{theorem}\label{thm:welldefinedness}
Let \cref{ass:F_regularity} hold. Let
$\nabla F(w^n) \in C(\bar{\Omega})$.
Then the loop starting in \cref{alg:slipsub_greedy} Line 
\ref{ln:tabulation_loop} terminates after finitely many 
iterations and $\calA$ is not empty if $w^n$ is not stationary.
\end{theorem}
\begin{proof}
We need to prove that eventually $\calW = \emptyset$. We first observe that for every choice of
$k$ in the loop starting in Line \ref{ln:tabulation_loop}, the set $\calW$ contains at most
$|\calD|$ elements so that the loop starting in Line \ref{ln:inner_tabulation_loop} always
terminates finitely. In addition, elements are never removed but only added to the set $\calA$.

We briefly argue that by Lemma 5.2, \cref{alg:slipsub_greedy} Line \ref{ln:suff_dec}
will be executed at least once over the course of the inner loop of Line \ref{ln:tabulation_loop}.
Specifically, applying \cref{lem:instationarity} with the choice $w^\ell \coloneqq w^n$ for all $\ell \in \N$ gives
a tuple $(k_1,D_1)$ such that the acceptance criterion in Line \ref{ln:suff_dec} holds.
Consequently, there is some patch $D$ such that for increasing $k$ Line \ref{ln:increase_k}
is executed and $(k,D) \in \calW$ holds until the acceptance criterion in Line \ref{ln:suff_dec} or 
$\max_{(\tilde{k},\tilde{D}) \in \calA} \ared^{n,\tilde{k},\tilde{D}}
\ge \pred^{n,k,D} + L_{\nabla F} \Delta^{0}
2^{-k}$ holds. In both cases, $\calA$ is not empty after iteration $k$.

We also observe that the $\max$ in Line \ref{ln:not_pred_zero_or_ared_domination}
is strictly greater than zero from this point on because pairs $(k,D)$ are only inserted
into $\calA$ if $\pred^{n,k,D}$ and in turn also $\ared^{n,k,D}$ are strictly positive.

We proceed by contradiction and assume that the algorithm does not terminate finitely so
that there is a patch $D_2$ such that
for all $k \in \N$, the criterion in Line \ref{ln:not_pred_zero_or_ared_domination} is satisfied
while the criterion in Line \ref{ln:suff_dec} is not satisfied.
In combination with $\calA$ eventually being non-empty and the left hand side of
the $\max$ in Line \ref{ln:not_pred_zero_or_ared_domination} being strictly positive,
this means that $\pred^{n,k,D_2}$ is uniformly bounded away from zero for all $k \in \N$.

Clearly, for all feasible points of
$\text{{\ref{eq:trp}}}(w^{n}, \nabla F(w^{n}), D_2, \Delta^{0} 2^{-k})$ and thus the minimizer
$\tilde{w}^{n,k,D_2}$ it holds
\begin{multline*}
J(w^n) - J(\tilde{w}^{n,k,D_2}) = (\nabla F(w^n), w^n - \tilde{w}^{n,k,D_2})_{L^2} + \alpha \TV(w^n) - \alpha \TV(\tilde{w}^{n,k,D_2})\\
+ (\nabla F(\xi^{n,k,D_2}) - \nabla F(w^n), w^n - \tilde{w}^{n,k,D_2})_{L^2}
\end{multline*}
for some $\xi^{n,k,D_2}$ in the line segment between $w^n$ and $\tilde{w}^{n,k,D_2}$.
Moreover,
\[
|(\nabla F(\xi^{n,k,D_2}) - \nabla F(w^n), \tilde{w}^{n,k,D_2} - w^n)_{L^2}|
\le L_{\nabla F} \Delta^0 2^{-k}
\]
holds by means of \cref{ass:F_regularity} \ref{itm:nablaF_Lipschitz}. In turn, we obtain
\begin{gather*}
\frac{\ared^{n,k,D_2}}{\pred^{n,k,D_2}}
= \frac{\pred^{n,k,D_2} + (\nabla F(\xi^{n,k,D_2}) - \nabla F(w^n), w^n - \tilde{w}^{n,k,D_2})_{L^2}}{\pred^{n,k,D_2}}
\to 1 > \sigma
\end{gather*}
so that eventually Line \ref{ln:suff_dec} holds for $D_2$ and some large enough $k$ too,
thereby contradicting that the criterion in Line \ref{ln:suff_dec} is never satisfied
for $D_2$.
\end{proof}
\begin{remark}
A close inspection of the argument of \cref{thm:welldefinedness} shows that it is possible to
avoid \cref{ass:F_regularity} \ref{itm:nablaF_Lipschitz} in \cref{thm:welldefinedness}
and thus in our overall arguments by replacing $L_{\nabla F}$ in \cref{alg:slipsub_greedy}
by $2c(b)$ with $c(b)$ from \eqref{eq:cofb} if a bound $b > \TV(w^{n,k,D})$ can be established
uniformly for all iterates $w^{n,k,D}$. Such bounds exist due to the properties
of $F$ and the descent properties of the algorithm. If 
\[ c(\infty) = \sup\{ \|\nabla F(w)\|_{L^\infty}\,|\, w(x) \in W \text{ a.e.} \} < \infty \]
holds, this is of course sufficient too.
\end{remark}

\begin{lemma}\label{lem:nonincreasing_objectives_weakstar_accumulation_point}
Let $\{w^n\}_{n}$ be the sequence of iterates produced by
\Cref{alg:slipsub_greedy}. Then the sequence of objective values
$\{J(w^n)\}_n$ is monotonously
non-increasing and convergent.
The sequence $\{w^n\}_{n}$ admits a feasible weak-$^*$ accumulation point
in $\BVW(\Omega)$.
\end{lemma}
\begin{proof}
    The sequence of objective values is monotonically nonincreasing by
construction and, because $F$ and $\TV$ are both bounded below,
also convergent. Because $\{\TV(w^n)\}_n$ is uniformly bounded
above by $J(w^0) < \infty$, $\{w^n\}_n$ admits a weak-$^*$
accumulation point $\bar{w}$, which is feasible, that
is in $\BVW(\Omega)$.
\end{proof}
If $\bar{w}$ is also a strict accumulation point, that is $\TV(w^{n_k}) \to \TV(\bar{w})$
in addition to $w^{n_k} \weakstarto \bar{w}$ in $\BV(\Omega)$ holds for a subsequence $\{w^{n_k}\}_k$,
then $J(w^{n_k}) \to J(\bar{w})$ under a continuity assumption on $F$.
Consequently, we desire that every weak-$^*$ accumulation point is a strict accumulation point.
\begin{lemma}\label{lem:weakstar_accumulation_points_are_strict}
Let \cref{ass:F_regularity,ass:patches} hold and $\{w^n\}_n$ be the sequence of iterates produced by
\cref{alg:slipsub_greedy}. Every weak-$^*$ accumulation point $\bar{w}$ of $\{w^n\}_n$ is a
strict accumulation point and $J(w^n) \to J(\bar{w})$.
\end{lemma}
\begin{proof}
Let $\bar{w}$ be a weak-$^*$ accumulation point in $\BV(\Omega)$ with approximating subsequence
$\{w^{n_\ell}\}_{\ell}$. Then $w^{n_\ell} \to \bar{w}$ in $L^1(\Omega)$ and in turn
$F(w^{n_\ell}) \to F(\bar{w})$.

\textbf{Outline.} By way of contradiction,
we assume that $\bar{w}$ is not strict, that is, after
a possible restriction to a sub-subsequence,
\begin{gather}\label{eq:nonstrictwstarconv_inequality}
\lim_{\ell\to\infty} \TV(w^{n_\ell}) > \TV(\bar{w}) + \varepsilon
\end{gather}
for some $\varepsilon > 0$.
We are going to localize this inequality, that is we show that there exists an
open ball $B$ such that $\overline{B} \subset D$ for some $D \in \calD$, and $\varepsilon_B > 0$ such that
\begin{gather}\label{eq:nonstrictwstarconv_patch_inequality}
\lim_{\ell\to\infty} \TVB(w^{n_\ell}) > \TVB(\bar{w}) + \varepsilon_B
\end{gather}
and whose boundary does not interfere with the interfaces of the level sets
of the $w^{n_\ell}$ and $\bar{w}$. Specifically, we require
\begin{gather}\label{eq:no_interface_contribution}
\Ha^{d-1}\left(\partial B \cap \bigcup_{i = 1}^M \partial^* E_i^{n_\ell}\right) = 0
\text{ for all } \ell \in \N
\quad\text{ and }\quad
\Ha^{d-1}\left(\partial B \cap \bigcup_{i = 1}^M \partial^* E_i\right) = 0,
\end{gather}
where $\{E_1,\ldots,E_M\}$ is a Caccioppoli partition of $\Omega$ such that
$\bar{w} = \sum_{i=1}^M w_i\chi_{E_i}$ and similarly for $w^{n_\ell}$, see \cite[Lemma 2.1]{manns2023on}.
With this localization, we will then be able to construct competitors that are feasible
for trust-region subproblems \eqref{eq:trp} on the patch $D$ and guarantee a
predicted and actual reduction that are bounded away from zero. In turn, we will obtain
$J(w^n) \to -\infty$, which gives the desired contradiction.

\textbf{Localization of \eqref{eq:nonstrictwstarconv_inequality} to
\eqref{eq:nonstrictwstarconv_patch_inequality}.}
In order to obtain the existence of $B$ and $D$, we employ a covering argument and consider
the following set of closed balls:
\begin{gather}\label{eq:fine_cover}
\mathcal{F} = \left\{ \overline{B_s(x)} \,\middle|\,
\begin{aligned}
x \in \Omega, 0 < s, D \in \calD, \overline{B_s(x)} \subset D,\\
\eqref{eq:no_interface_contribution} \text{ holds for the choice } \overline{B} = \overline{B_s(x)}
\end{aligned}
\right\}.
\end{gather}
Clearly, \eqref{eq:fine_cover} would be a fine cover if \eqref{eq:no_interface_contribution} is not required
for its elements. Because there are only countably many iterates $w^n$, there are only countably many
Caccioppoli partitions $\{E_1^{n_\ell},\ldots,E_M^{n_\ell}\}$. Consequently, we can always perturb $s$ slightly
(arbitrarily small if needed) to ensure both identities in \eqref{eq:no_interface_contribution} hold for a closed
ball $\overline{B} = \overline{B_r(x)}$, see also \cite[Proposition 2.16]{maggi2012sets}. This means that for all
$r_0 > 0$, we find some $r \in (0,r_0)$ such that
$\overline{B_r(x)} \subset \overline{B_{r_0}(x)}$ and $\overline{B_r(x)} \subset D$.

The Vitali--Besicovitch covering theorem implies that there is a countable and pairwise disjoint subset
$\tilde{\mathcal{F}} \subset \mathcal{F}$ such that
$|Dw| = \sum_{i=1}^{M-1}\sum_{j=i+1}^M|w_i - w_{j}|\Ha^{d-1} \mres (\partial^* E_i \cap \partial^* E_j \cap \Omega)$
satisfies
$|Dw|(\Omega \setminus \bigcup_{\overline{B} \in \tilde{\mathcal{F}}} \overline{B}) = 0$
and $|Dw^{n_\ell}|(\Omega \setminus \bigcup_{\overline{B} \in \tilde{\mathcal{F}}} \overline{B}) = 0$
for all $\ell \in \N$.

In combination with the fact that \eqref{eq:no_interface_contribution} holds for all balls in
$\mathcal{F}$ and thus also $\tilde{\mathcal{F}}$, we obtain \eqref{eq:nonstrictwstarconv_patch_inequality}
for some open ball $B$ with $\overline{B} \subset \Omega$ and $\varepsilon_B > 0$.

\textbf{Competitor construction.}
We seek competitors $\hat{w}^{n_\ell}$, which eventually become feasible for the problem
$\text{{\ref{eq:trp}}}(w^{n_\ell}, \nabla F(w^{n_\ell}), D, \Delta^0 2^{-k})$ for large enough $\ell$
and suitable corresponding $k$. Let $r > 0$ and $\bar{x} \in D$ satisfy $B = B_r(\bar{x})$ for the above-asserted
$B$ and $D$. Moreover, let $r_1 > r$ be small enough such that $\overline{B_{r_1}(\bar{x})} \subset D$ holds.
It is possible to find such $r_1$ due to the strictly positive distance of $B$ to the boundary $\partial D$
(note that $D$ is open and $\overline{B}$ is closed). We apply \cref{lem:competitor} with the choices
$\hat{w} = \bar{w}$ and $K = \overline{B_{r_1}(\bar{x})}$. We obtain that there exist
$D_s \subset \subset D$ with $\overline{B_{r_1}(\bar{x})} \subset \subset D_s$ and
functions $\{\hat{w}^{n_\ell}\}_\ell$ with $\hat{w}^{n_\ell}(x) = w^{n_\ell}(x)$ for
$x \in \Omega\setminus D_s$ and $\hat{w}^{n_\ell}(x) = \bar{w}(x)$ for $x \in D_s$ such that
\begin{align*}
\hat{w}^{n_\ell} &\weakstarto \bar{w} \text{ in } \BVW(\Omega) \text{ for } \ell \to \infty,\\
\TV(\hat{w}^{n_\ell}) &= {\TV}_{\Omega\setminus \overline{D_s}}(w^{n_\ell})
+ {\TV}_{D_s}(\bar{w}) + b^\ell\text{ with } b^\ell \ge 0 \text{ and } \liminf_{\ell\to\infty} b^\ell = 0,\\
\TV(\bar{w}) &= {\TV}_{\Omega\setminus \overline{D_s}}(\bar{w}) + {\TV}_{D_s}(\bar{w}).
\end{align*}
Using \eqref{eq:nonstrictwstarconv_patch_inequality} and \eqref{eq:no_interface_contribution}, we obtain
\begin{align*}
\TV(\hat{w}^{n_\ell}) 
&= {\TV}_{\Omega\setminus \overline{D_s}}(w^{n_\ell}) + {\TV}_{D_s}(\bar{w}) + b^\ell\\
&= {\TV}_{\Omega\setminus \overline{D_s}}(w^{n_\ell}) + {\TV}_{D_s\setminus B}(\bar{w}) + {\TV}_{B}(\bar{w}) + b^\ell\\
&< {\TV}_{\Omega\setminus \overline{D_s}}(w^{n_\ell}) + \lim_{\ell \to \infty} {\TV}_{D_s\setminus B}(w^{n_\ell})
+ {\TV}_{B}(w^{n_\ell}) + b^\ell - \varepsilon_B.
\end{align*}
Taking the $\liminf$ over $\ell$ on both sides and using \eqref{eq:no_interface_contribution} again,
we obtain
\begin{gather}\label{eq:competitor_improvement}
\TV(\bar{w}) \le \liminf_{\ell \to \infty} \TV(\hat{w}^{n_\ell})  \le \lim_{\ell\to\infty} \TV(w^{n_\ell}) - \varepsilon_B.
\end{gather}

\textbf{Contradiction $J(w^n) \to -\infty$.}
We follow the arguments of the proof of Theorem 6.4 in \cite{manns2023on}
(namely Outcome 3, part 2) in order to show that the competitors
constructed above eventually
become feasible and enforce a reduction of the objective
that is bounded strictly away from zero infinitely often, which in turn
contradicts that the objective is bounded below.

Let $\delta \coloneqq \frac{\varepsilon_B}{2}$. Because of \cref{ass:F_regularity}
and $\hat{w}^{n_\ell} \to \bar{w}$ in $L^2(\Omega)$ as well as $w^{n_\ell} \to \bar{w}$ in $L^2(\Omega)$,
we obtain that there exist some large enough $k_0 \in \N$ and $\ell_0 \in \N$ and such that 
by virtue of \eqref{eq:competitor_improvement} for infinitely many $\ell \ge \ell_0$ and all $\tilde{w} \in \BVW(\Omega)$
with $\|\tilde{w} - w^{n_\ell}\|_{L^1} \le \Delta^{0}2^{-k_0}$ it holds that
\begin{align}
|F(w^{n_\ell}) - F(\tilde{w})|
&\le \frac{1-\sigma}{3 - \sigma}\alpha\delta,\label{eq:F_lk_approx}\\
|(\nabla F(w^{n_\ell}), w^{n_\ell} - \tilde{w})_{L^2}|
&\le \frac{1-\sigma}{3 - \sigma}\alpha\delta, \label{eq:nablaF_lk_approx} \\
\|\hat{w}^{n_\ell} - w^{n_\ell}\|_{L^1} &\le \Delta^{0}2^{-k_0},\text{ and}\label{eq:w_lk_feasible}\\
\alpha \TV(w^{n_\ell}) - \alpha \TV(\hat{w}^{n_\ell}) &\ge \alpha\delta, \label{eq:tv_lk_diff}
\end{align}
where \eqref{eq:w_lk_feasible} gives that $\hat{w}^{n_\ell}$ is feasible for
$\text{{\ref{eq:trp}}}(w^{n_\ell}, \nabla F(w^{n_\ell}), D, \Delta^{0} 2^{-k_0})$.

Because $\pred{}^{n_\ell,k_0,D}$ is computed 
in \cref{alg:slipsub_greedy} Line \ref{ln:parallel_pred} as the negation of the objective value of a
minimizer of $\text{{\ref{eq:trp}}}(w^{n_\ell}, \nabla F(w^{n_\ell}), D, \Delta^{0} 2^{-k_0})$,
the feasibility \eqref{eq:w_lk_feasible} 
and the objective term estimates \eqref{eq:nablaF_lk_approx} and \eqref{eq:tv_lk_diff}
yield
\begin{align}
\pred{}^{n_\ell,k,D}
&\ge (\nabla F(w^{n_\ell}), w^{n_\ell} - \hat{w}^{n_\ell})_{L^2}
     + \alpha \TV(w^{n_\ell}) - \alpha \TV(\hat{w}^{n_\ell})
\ge - \frac{1-\sigma}{3 - \sigma}\alpha\delta + \alpha\delta.
\label{eq:pred_lk_lower_bound}
\end{align}
Moreover, \eqref{eq:F_lk_approx} and \eqref{eq:nablaF_lk_approx} give
\[
\ared{}^{n_\ell,k_0,D} \ge \pred{}^{n_\ell,k_0,D} - 2 \frac{1 - \sigma}{3 - \sigma}\alpha\delta,
\]
where the right-hand side is strictly positive because of
\eqref{eq:pred_lk_lower_bound} and $1 - \tfrac{1 - \sigma}{3 - \sigma}
> 2\tfrac{1 - \sigma}{3 - \sigma}$. Consequently, for suitable $\ell$ it holds that
\[
\frac{\ared{}^{n_\ell,k_0,D}}{\pred{}^{n_\ell,k_0,D}}
\ge
\frac{\pred{}^{n_\ell,k_0,D}
	- 2\frac{1 - \sigma}{3 - \sigma} \alpha\delta}{\pred{}^{n_\ell,k_0,D}}
\ge
\frac{1 - 3\frac{1 - \sigma}{3 - \sigma}}{1 - \frac{1 - \sigma}{3 - \sigma}}
= \sigma,
\]
where the second inequality follows from the monotonocity of the function
$p \mapsto p^{-1}(p - 2\tfrac{1 - \sigma}{3 - \sigma})$.

Consequently, for infinitely many $\ell \ge \ell_0$, the loop starting at
\cref{alg:slipsub_greedy} Line \ref{ln:tabulation_loop} terminates after finitely
many iterations with $(\bar{k},\bar{D}) \in \argmax\{ \ared^{n_\ell,\bar{k},\bar{D}}\,|\, (\bar{k},\bar{D}) \in \calA\}$
such that
\[ {\ared}^{n_\ell,\bar{k},\bar{D}} \ge {\ared}^{n_\ell,k_0,D} \ge \sigma {\pred}^{n_\ell,k_0,D}
   \ge 2\frac{1 - \sigma}{3 - \sigma}\sigma\alpha\delta
\]
Since the change induced by the maximizer $(\bar{k},\bar{D})$ will definitely be applied in
Line \ref{ln:apply_reduction} and all further reductions decrease the objective even more, we obtain
\begin{align*}
F(w^{n_\ell+1}) + \alpha \TV(w^{n_\ell+1})
- F(w^{n_\ell}) - \alpha \TV(w^{n_\ell})
&\ge {\ared}^{n_\ell,\bar{k},\bar{D}}  \ge \sigma 2 \alpha\frac{1-\sigma}{3-\sigma}\frac{\varepsilon_B}{2} > 0.
\end{align*}
Because the sequence $\{J(w^n)\}_n$ decreases monotonically, we obtain the contradiction
$J(w^{n}) \to -\infty$.

\textbf{Objective value convergence.} $F$ is continuous and $w^{n_\ell} \weakstarto \bar{w}$
in $\BVW(\Omega)$ implies $w^{n_\ell} \to \bar{w}$ in $L^1(\Omega)$ and
in turn $w^{n_\ell} \to \bar{w}$ in $L^2(\Omega)$ because $\BVW(\Omega)$ is bounded in $L^\infty(\Omega)$.
In combination with
the convergence of ${J(w^n)}_n$ and the strict convergence of $\{\TV(w^{n_\ell})\}_\ell\}$,
we obtain $J(w^{n_\ell}) \to J(\bar{w}) = \inf_{n \in \N} J(w^n)$.
\end{proof}

\begin{theorem}\label{thm:convergence}
Let \cref{ass:calD_open_cover,ass:patches,ass:F_regularity} hold.
Let $\{w^n\}_{n}$ be the sequence of iterates produced by \Cref{alg:slipsub_greedy}.
Then one of the following mutually exclusive outcomes holds:
\begin{enumerate}
\item\label{itm:outcome_outer_finite_inner_finite} The sequence $\{w^n\}_n$ is finite and the final element $w^N$
for some $N \in \N$ satisfies the following. For all $D \in \calD$
there exists $k \in \N$ such that $w^N$ solves
$\textrm{\emph{\ref{eq:trp}}}(w^{N}, \nabla F(w^{N}), D, \Delta^0 2^{-k})$.
In particular, $w^N$ is stationary if $\nabla F(w^N) \in C(\bar{\Omega})$.
\item\label{itm:outcome_outer_finite_inner_infinite} The sequence $\{w^n\}_n$ is finite and the final element $w^N$
for some $N \in \N$ satisfies the following. The loop over $k$ that
begins in Line \ref{ln:tabulation_loop} does not terminate finitely. In particular, $w^N$ is stationary if $\nabla F(w^N) \in C(\bar{\Omega})$.
\item\label{itm:outcome_outer_infinite} The sequence $\{w^n\}_n$ has a weak-$^*$ accumulation point
in $\BV(\Omega)$. All weak-$^*$ accumulation points are in
$\BVW(\Omega)$. If a weak-$^*$ accumulation point $\bar{w}$ satisfies
$\nabla F(\bar{w}) \in C(\bar{\Omega})$, it is stationary.
\end{enumerate}
\end{theorem}
\begin{proof}
We follow the basic proof strategy of Theorem 4.23 in \cite{leyffer2022sequential}
and Theorem 6.4 in \cite{manns2023on} and extend it so that we can perform
a localization to patches. We assume that Outcomes
\ref{itm:outcome_outer_finite_inner_finite} and
\ref{itm:outcome_outer_finite_inner_infinite} do not hold and prove that
Outcome \ref{itm:outcome_outer_infinite} must hold in this case.

\Cref{lem:nonincreasing_objectives_weakstar_accumulation_point} gives the existence of a weak-$^*$
accumulation point $\bar{w}$ with approximating subsequence of iterates $\{w^{n_\ell}\}_\ell$. Because
weak-$^*$ convergence implies pointwise a.e.\ convergence for a subsequence, we have
$\bar{w} \in \BVW(\Omega)$. Now let $\nabla F(\bar{w}) \in C(\bar{\Omega})$.
\Cref{lem:weakstar_accumulation_points_are_strict} gives
$\TV(w^{n_\ell}) \to \TV(\bar{w})$ and $F(w^{n_\ell}) \to F(\bar{w})$.
We proceed by way of contradiction and assume that $\bar{w}$ is not stationary.

We apply \cref{lem:instationarity} and obtain that there are $D \in \calD$, $k_0 \in \N$, $\ell_0 \in \N$,
and $\varepsilon > 0$ such that
\[ {\ared}^{n_\ell,k_0,D} \ge \sigma {\pred}^{n_\ell,k_0,D}
\enskip\text{and}\enskip
{\pred}^{n_\ell,k_0,D} > \varepsilon. \]
Consequently, for infinitely many $\ell \ge \ell_0$, the loop starting at
\cref{alg:slipsub_greedy} Line \ref{ln:tabulation_loop} terminates after finitely
many iterations with $(\bar{k},\bar{D}) \in \argmax\{ \ared^{n_\ell,\bar{k},\bar{D}}\,|\, (\bar{k},\bar{D}) \in \calA\}$
such that
\[ {\ared}^{n_\ell,\bar{k},\bar{D}} \ge {\ared}^{n_\ell,k_0,D} \ge \sigma {\pred}^{n_\ell,k_0,D}
\ge \sigma \varepsilon.
\]
Since the change induced by the maximizer $(\bar{k},\bar{D})$ will definitely be applied in
Line \ref{ln:apply_reduction} and all further reductions decrease the objective even more, we obtain
\begin{align*}
F(w^{n_\ell+1}) + \alpha \TV(w^{n_\ell+1})
- F(w^{n_\ell}) - \alpha \TV(w^{n_\ell})
&\ge {\ared}^{n_\ell,\bar{k},\bar{D}}  \ge \sigma \varepsilon > 0.
\end{align*}
Note that the while loop stops at the first instance we encounter an increase after the guaranteed
first decrease or the set $\calA$ becomes empty eventually so that we always get finite termination of
this loop. Because the sequence ${J(w^n)}_n$ decreases monotonically, we obtain the contradiction
$J(w^{n}) \to -\infty$.
\end{proof}

\section{Numerical Experiments}\label{sec:numerics}
This section performs numerical tests to compare the block-SLIP algorithm \Cref{alg:slipsub_greedy}
results to its primary competitor, SLIP \cite{leyffer2022sequential,manns2023on}. Note that block-SLIP is equivalent
to SLIP if the number of patches is equal to one.

The goal of our numerical study is to assess if and when the block-SLIP algorithm can outperform
SLIP and help to scale the problem sizes. To this end, we consider two test cases. First,
we consider an integer control problem that is defined on a one-dimensional domain, where we can
use the efficient subproblem solver from \cite{severitt2023efficient}. Second, we consider an integer
control problem on a two-dimensional domain, where no efficient subproblem solver is known and we need to
resort to an off-the-shelf integer programming solver.

Throughout this section, results of block-SLIP instances are designated by the subscript $\emph{bs}$ and
results of SLIP instances are designated by the subscript $\emph{s}$.
We have carried out the experiments on a node of the Linux HPC cluster LiDO3 with two
AMD EPYC 7542 32-Core CPUs and 64 GB RAM (computations were restricted to one CPU).

\subsection{1D Test Case}
We solve the integer optimal control benchmark problem that is defined in (5.1) in \cite{leyffer2022sequential} on 
the one-dimensional domain $\Omega = (-1,1)$ with the discretization reported therein. Specifically, the problem
reads
\begin{gather}\label{eq:example_1d}
\min_{w}\, \frac{1}{2}\|K w - f\|_{L^2(\Omega)}^2 + \alpha \TV(w)\enskip\textup{s.t.}\enskip w(t) \in W\text{ for a.e.\ } t \in \Omega,
\end{gather}
where we make the choices $W = \{-1,0,1\}$ and $f(t) = 0.2 \cos(2 \pi t - 0.25) \exp(t)$ for $t \in (-1,1)$ and have
$Kw = k * w$ with $(k*w)(t) = \int_{-1}^t k(t - \tau)w(\tau)\dd \tau$ with $k$ as in \S5 in \cite{leyffer2022sequential}.
It is a deconvolution problem that stems from \emph{Filtered Approximation} in electronics. It was analyzed as a
finite-dimensional \emph{convex quadratic integer program} in \cite{buchheim2012effective} and different variants have
been used as benchmark problems for integer optimal control algorithms in
\cite{kirches2021compactness,leyffer2022sequential,marko2023integer,severitt2023efficient}.

\paragraph{Subproblem Solution}
We use the topological sorting-based algorithm from \cite{severitt2023efficient} as subproblem solver since
(our latest implementation\footnote{\url{https://github.com/paulmanns/trs4slip}}) outperforms the $A^*$-based and
integer programming-based solution approaches that are also discussed in \cite{severitt2023efficient} on our
benchmark problems significantly.

\paragraph{Benchmark}
For our experiments, we make the choices of $N \in \{2^{12}, 2^{14}\}$ for the number of discretization intervals for the
piecewise constant ansatz for the control input function and $\alpha \in \{1.25\cdot 10^{-4}, 5.0\cdot 10^{-4}, 2.0 \cdot 10^{-3} \}$.
We cover a spectrum of values for $\alpha$ because the the numbers of iterations that SLIP and block-SLIP require can differ
substantially for different values $\alpha$. We highlight that, in contrast to an integer programming solver as is used for the second 
test case, the performance of the topological sorting-based subproblem solver only depends on the current value of the trust-region 
radius and not on the data (current iterate), see also \cite{severitt2023efficient}. As initial value for the optimization, we choose 
the control $w^0 \equiv 0$ for all instances.

For all of these instances, we run SLIP, block-SLIP with $N_p = 4$ patches, and block-SLIP with $N_p = 9$ patches.
The patches are uniform intervals, uniformly distributed over the domain, and always overlap by $0.2$ with their left and
right neighbors.

\paragraph{Algorithm Setup}
Regarding the algorithm, we choose $\Delta_0 = 0.125$ and $\sigma = 10^{-4}$. We can determine contraction of
the trust-region when $\Delta$ falls below the volume of one grid cell so that, due its discreteness, the subproblem
solution coincides with the previous iterate and no progress is possible. Then we terminate the algorithm.
We also prescribe a limit of $1000$ outer iterations but note that all SLIP and block-SLIP runs on our instances
terminate due to a contraction of the trust-region radius/no further progress being possible.

\paragraph{Results}
SLIP and block-SLIP have generally returned points with similar objective values. Specifically, the objective values
for all instances with $\alpha \in \{5.0 \cdot 10^{-4}, 2.0\cdot 10^{-3}\}$ are very close; they differ by less than $1\,\%$.
Only for the instance with $\alpha = 1.25 \cdot 10^{-4}$, block-SLIP returns a better objective value than SLIP,
which is approximately $8.7\,\%$ lower than the one produced by SLIP (consistently across discretizations and number of patches).
This behavior is not unexpected since SLIP and block-SLIP are local optimization techniques that compute different steps that may
converge to different (stationary) points. We highlight that due to the nonconvexity and the nontrivial
nature of the optimality condition we do not know how many stationary points the problem has and how this information could
straightforwardly be obtained.
Regarding run times, SLIP outperforms block-SLIP by a wide margin. This is due to the run time complexity of
the topological sorting-based subproblem solver, which is $O(\tilde{N} \tilde{K} M)$. Here, $\tilde{N}$ is
equal to $N$ for SLIP and is equal to the number of intervals of the patch problem for block-SLIP, which is
larger than $N / N_p$ because of the overlap of the patches; $\tilde{K} = \max\{\Delta N, \tilde{N}\}$,
which is equal to $\Delta N$ on all of our instances for both SLIP and block-SLIP. As a consequence, if
the same number of outer iterations is executed and the steps on the different patches become
acceptable for the same trust-region radius, the run time spent in the subproblem solver for block-SLIP must be higher
than for SLIP. We observe that block-SLIP generally requires more outer iterations on our benchmark, increasing
the run time spent in the subproblem solver even more. Moreover, many additional evaluations of the control-to-state
operator are necessary in block-SLIP, increasing the total run time of block-SLIP further.
A detailed tabulation of the objective values and run times is given in \cref{tab:exp2}.

In conclusion, our results for the first test case show that block-SLIP does not pay off for one-dimensional problems
since the subproblem solver is extremely efficient and scales well irrespective of the specific data.
The only reason that we can sensibly think of using block-SLIP in 1D are extreme cases, where very high values
of $N$ are required because the subproblem solver runs out of memory at some point.
For our test case and compute environment with 64\,GB RAM, this happens at $N = 2^{16}$.
 
\begin{table}[ht]
	\centering
	\caption{Objective values $J(x)$ and broken down to $f(x)$ and $\TV(x)$ and run times $t_x$ for $x = x_s$ (solution returned by SLIP)
	and $x = x_{bs}$ (solution returned by block-SLIP) for the different instances of our one-dimensional benchmark problem.
	In each row, the winner(s) in terms of objective and run time up to the reported precision are
	highlighted with bold-faced text.}\label{tab:exp1}	
	\begin{adjustbox}{width=\textwidth}	    
	\begin{tabular}{lllllllllll}
	\toprule
     $N$ & $N_{\rm p}$ & $\alpha\cdot 10^{-3}$ & $J(x_{\rm bs})$ & $J(x_{\rm s})$ & $f(x_{\rm bs})$ & $\TV(x_{\rm bs})$ & $f(x_{\rm s})$ & $\TV(x_{\rm s})$ & $t_{\rm bs}$ & $t_{\rm s}$ \\
	\midrule
	\multirow[t]{8}{*}{12} & \multirow[t]{4}{*}{4} 
	   & 0.125 & \textbf{0.002757} & 0.003017 & 0.001257 & 12 & 0.001517 & 12 & 41 & \textbf{21} \\
	&  & 0.500 & \textbf{0.006074} & \textbf{0.006074} & 0.002074 & 8 & 0.002074 & 8 & 27 & \textbf{13} \\
	&  & 2.000 & 0.015788 & \textbf{0.015787} & 0.003788 & 6 & 0.003787 & 6 & 13 & \textbf{6} \\
	\cline{2-11}
	& \multirow[t]{4}{*}{9} 
	   & 0.125 & \textbf{0.002759} & 0.003017 & 0.001259 & 12 & 0.001517 & 12 & 135 & \textbf{21} \\
	&  & 0.500 & \textbf{0.006074} & \textbf{0.006074} & 0.002074 & 8 & 0.002074 & 8 & 54 & \textbf{13} \\
	&  & 2.000 & \textbf{0.015787} & \textbf{0.015787} & 0.003787 & 6 & 0.003787 & 6 & 22 & \textbf{6} \\
	\cline{1-11} \cline{2-11}
	\multirow[t]{8}{*}{14} & \multirow[t]{4}{*}{4} 
	   & 0.125 & \textbf{0.002743} & 0.003010 & 0.001243 & 12 & 0.001510 & 12 & 1579 & \textbf{615} \\
	&  & 0.500 & \textbf{0.006072} & \textbf{0.006072} & 0.002072 & 8 & 0.002072 & 8 & 636 & \textbf{267} \\
	&  & 2.000 & \textbf{0.015786} & \textbf{0.015786} & 0.003786 & 6 & 0.003786 & 6 & 461 & \textbf{119} \\
	\cline{2-11}
	& \multirow[t]{4}{*}{9} 
	   & 0.125 & \textbf{0.002744} & 0.003010 & 0.001244 & 12 & 0.001510 & 12 & 4671 & \textbf{615} \\
	&  & 0.500 & \textbf{0.006072} & \textbf{0.006072} & 0.002072 & 8 & 0.002072 & 8 & 1659 & \textbf{267} \\
	&  & 2.000 & \textbf{0.015786} & \textbf{0.015786} & 0.003786 & 6 & 0.003786 & 6 & 432 & \textbf{119} \\
	\cline{1-11} \cline{2-11}
	\bottomrule
	\end{tabular}
	\end{adjustbox}
\end{table}

\subsection{2D Test Case}\label{sec:2Dtestcase}
As our 2D test case, we choose $W = \{0,1\}$ and a convection-diffusion equation that is specified as follows.
We consider the square domain $\Omega = (0,1)^2$ and for a given control $w$ with $w(x) \in  W$ a.e.,
the state vector $u$ is given by the solution to
\begin{gather}\label{eq:pde2d}
\begin{aligned}
-\varepsilon \Delta u + c_1 \cdot \nabla u + c_2 u w &= f \quad\text{in } \Omega \\
u&= 0 \quad\text{on } \{0,1\} \times (0,1) \cup ((0,0.25) \cup (0.75,1)) \times \{0\} \\
u&= \sin(2 \pi (x_1 - 0.25)) \quad\text{on } (0.25,0.75) \times \{0\} \\
\partial_n u &= 0 \quad\text{on } (0,1) \times \{1\},
\end{aligned}
\end{gather}
where $c_2 = 2$, $c_1(x) = (\begin{matrix} \sin(\pi x_1)
& \cos(2 \pi x_2)\end{matrix})^T$ for $x \in \Omega$, 
$f(x) = \sin(2 \pi x_1 + 2 \pi x_2) + 3$ for $x \in \Omega$, 
and $\varepsilon = 4\cdot10^{-2}$.

Let the solution operator to \eqref{eq:pde2d} be denoted by $S$. We choose the objective
\[ F = j\circ S \enskip\text{ with }\enskip j(u) = \frac{1}{2}\|u - u_d\|_{L^2}^2, \]
where $u_d$ is computed by solving a variant of \eqref{eq:pde2d}, where $c_1$ 
is replaced by $\tilde{c_1}(x) = (\begin{matrix} -x_2
& 2 x_1\end{matrix})^T$, for $w = 2.5 \chi_{A} - 4(x_1 - 0.35)^3 \chi_A - 6(x_2 - 0.35)^3\chi_B$
with $A = (0,0.35)^2$ and $B = \Omega\setminus (0,0.35)^2$.
Note that we use a different PDE for computing $u_d$ to ensure that $u_d$ cannot be reached
or almost be reached and thus avoid that $F(u)$ can have values very close to zero, which are more difficult
to compare (with relative error computations which are more influenced by numerical errors when the denominator
becomes close to zero).

\paragraph{Discretization and Subproblem Solution}
In order to discretize the control and the PDE and to assemble the finite-element matrices for the PDE,
we use the finite-element package \texttt{FEniCSx} \cite{Dolfinx}, in which we
choose a piecewise constant control ansatz on a uniform $N \times N$ of square grid cells and solve the PDE
on the same grid with each grid of the cells being decomposed into 4 triangles, where we use continuous
Lagrange elements of order one for $u$ and $u_d$.

In order to compute $\nabla F(w)$ on the computer, we fix the discretization and then determine the adjoint
using operator calculus and the finite-element system described above so that we follow a
\emph{first-discretize, then-optimize} principle. 
We evaluate the total variation directly on the control functions as described in \cite{manns2023on}, which
may introduce an anisotropic effect on the resulting $w$, see the considerations in \cite{schiemann2024discretization}.
Since the convergence properties of the $\TV$-discretization and accuracy of the geometry of the resulting functions
are not the goal of our experiments, we find this reasonable to avoid much longer compute times that would
otherwise be necessary when using the convergent discretization scheme from \cite{schiemann2024discretization}.
Since this discretization is applied to all instances over all algorithms, we still achieve a fair comparison.
The resulting integer linear programs for the discretized trust-region subproblems are solved by means
of \texttt{Gurobi} \cite{Gurobi}; see \cite{manns2024discrete} for a detailed
MIP formulation and the accompanying
repository\footnote{\url{https://github.com/INFORMSJoC/2024.0680}}
for a possible implementation using \texttt{Gurobi}'s python API.

\paragraph{Benchmark}
For our experiments, we make the choices $N \in \{64, 96\}$ and
$\alpha \in \{5\cdot 10^{-4},  7.5\cdot 10^{-4},        10^{-3}, 1.25\cdot 10^{-3},
            1.5\cdot 10^{-3}, 1.75\cdot 10^{-3}, 2\cdot 10^{-3}, 2.25\cdot 10^{-3}\}$.
We cover this spectrum of values for $\alpha$ because small and large values of $\alpha$ generally
lead to inexpensive instances for integer programming solvers with a lot of chattering behavior in the resulting
functions for small $\alpha$ and basically constant functions
for large values of $\alpha$. Consequently, we cover several values in between, where the run times are relatively high.
Moreover, we solve one instance for $N = 128$ with $\alpha = 10^{-3}$. We restrict to one value of $\alpha$ here
because the run time was very high. As initial value for the optimization, we choose the control $w^0 \equiv 0$
for all instances.

For all of these instances, we run SLIP, block-SLIP with $N_p = 4$ patches, and block-SLIP with $N_p = 9$
patches. The patches are of uniform size, uniformly distributed over the domain,
and overlap by $0.1$ in each axis with the neighboring patches. This is visualized in \cref{fig:patch_overlap}.

\paragraph{Algorithm Setup}
Regarding the algorithm, we choose $\Delta_0 = 0.125$ and $\sigma = 10^{-4}$. We can determine contraction of
the trust-region when $\Delta$ falls below the volume of one grid cell so that, due its discreteness, the subproblem
solution coincides with the previous iterate and no progress is possible. Then we terminate the algorithm.
We prescribe a limit of $100$ outer iterations, which has never been reached in our experiments.

\paragraph{Results}
SLIP and block-SLIP have returned very similar objective values for all conducted experiments. The objective values equal to four digits of accuracy in almost all cases. In the three remaining cases, the point returned by SLIP is slightly better (relative improvement of objective value $< 1\%$). \Cref{fig:ctrl_visualization} gives a visual impression for the computed solutions for $\alpha = 10^{-3}$. There are essentially two different points the algorithm variants produced.

For $N = 64$, the run times of SLIP and block-SLIP are similar but generally low (the most expensive instance
for SLIP has a run time slightly less than 6 minutes). For $N = 96$, the comparison between SLIP
and block-SLIP is different for the numbers of patches $N_p = 4$ and $N_p = 9$. For $N_p = 4$,
there is no clear winner. 
While a run of block-SLIP $N_p = 4$ gives the highest absolute speedup with
more than $16750\,s$ run time improvement from $66740\,s$ to $49990\,s$, there is also an instance,
where block-SLIP takes $10807\,s$, which is almost twice than SLIP does on this instance with
$5628\,s$. The reason is that for $N = 96$ and $N_p = 4$, the integer programs
resulting from the discretized trust-region subproblems are already quite large on the different patches
and therefore can have relatively long compute times. Moreover, more subproblems are solved in
total due to the domain decomposition approach. 
This is different for the higher number of patches $N_p = 9$. Except for one inexpensive instance
for the smallest value of $\alpha$, the run time of block-SLIP is substantially lower compared to
SLIP because the run time of the subproblems drops substantially now.
For the two most expensive instances of SLIP with run times of $9151\,s$ and $66740\,s$,
the achieved speedups are $25.28$ and $7.4$ and the run times for block-SLIP are $362\,s$ and $9017\,s$.
For the instances with $N = 128$, the effects observed for $N = 96$ get amplified. The decomposition into
$N_p = 4$ patches is counterproductive and the run time increases from already very expensive $134313\,s$ for SLIP to
$393184\,s$ with block-SLIP for $N_p = 4$. Many trust-region subproblems of block-SLIP for $N_p = 4$ are very expensive
in this case. In contrast to this, the run time of block-SLIP drops to $1053\,s$ for $N_p = 9$, which is a
speedup of $127.55$. A detailed tabulation of the objective values and compute times is given in \cref{tab:exp2}.

The computational effort to solve the trust-region subproblems highly depends on the patch in
our example. For both $N_{\rm p} = 4$ and $N_{\rm p} = 9$, the bottom left patch (see \cref{fig:patch_overlap})
induces a much higher computational effort for the integer programming solver than all other patches. This
computational effort is further concentrated to instances with comparatively large trust-region radii.
We provide mean and median runtimes of the trust-region subproblems for the case $\alpha = 10^{-3}$
in \cref{tab:exp2cum4} for $N_{\rm p} = 4$ and \cref{tab:exp2cum9} for $N_{\rm p} = 9$. This observation
is reflected in the properties of the linear programming relaxation of the trust-region subproblem.
Specifically, the solution to the linear programming relaxation is already integer-valued and thus optimal
for the other patches in almost all instances. We note that it can be expected but not guaranteed
that the linear programming relaxation is integer-valued in large parts of the domain since Theorem 3 in
\cite{manns2024discrete} guarantees that there is at most one connected component of grid cells, where it
can be fractional. In such cases, when the number of patches is relatively small and the runtimes vary
greatly between the patches, we firmly believe that parallelization is advisable on the subproblem solver
instead on the level of \cref{alg:slipsub_greedy}.

In conclusion, our results show that block-SLIP pays off very well for large problem sizes and a
sufficiently large number of patches so that the integer programs can be solved quickly.
When compute times are already low, block-SLIP generally does not have a beneficial effect.
\begin{figure*}[t!]
	\centering
	\begin{subfigure}[t]{0.49\textwidth}
		\centering
		\begin{tikzpicture}
		\draw[black,semithick] (0,0) -- (5,0) -- (5,5) -- (0,5) -- (0,0);
		\draw[draw=none,fill = white, pattern={Lines[angle=-45,distance={6pt}]}]
			(0,0) -- (3,0) -- (3,3) -- (0,3) -- (0,0);
		\draw[draw=none,fill = white, pattern={Lines[angle=0,distance={6pt}]}]
			(2,0) -- (2,3) -- (5,3) -- (5,0) -- (2,0);
		\draw[draw=none,fill = white, pattern={Lines[angle=45,distance={6pt}]}]
			(2,2) -- (2,5) -- (5,5) -- (5,2) -- (2,2);
		\draw[draw=none,fill = white, pattern={Lines[angle=90,distance={6pt}]}]
			(0,2) -- (3,2) -- (3,5) -- (0,5) -- (0,2);
		\draw[black,very thick,dotted] (0,2) -- (5,2);
		\draw[black,very thick,dotted] (0,3) -- (5,3);
		\draw[black,very thick,dotted] (2,0) -- (2,5);
		\draw[black,very thick,dotted] (3,0) -- (3,5);			
		\node[anchor=north] at (0, 0) {$0.$};
		\node[anchor=north] at (2, 0) {$.4$};
		\node[anchor=north] at (3, 0) {$.6$};
		\node[anchor=north] at (5, 0) {$1.$};
		\node[anchor=east] at (0, 0) {$0.$};
		\node[anchor=east] at (0, 2) {$.4$};
		\node[anchor=east] at (0, 3) {$.6$};
		\node[anchor=east] at (0, 5) {$1.$};		
		\end{tikzpicture}
		\caption{$N_p = 4$ overlapping patches.}
	\end{subfigure}%
	~ 
	\begin{subfigure}[t]{0.49\textwidth}
		\centering
		\begin{tikzpicture}
		\draw[black,semithick] (0,0) -- (5,0) -- (5,5) -- (0,5) -- (0,0);
		\draw[draw=none,fill = white, pattern={Lines[angle=-45,distance={5pt}]}]
			(0,0) -- (2,0) -- (2,2) -- (0,2) -- (0,0);
		\draw[draw=none,fill = white, pattern={Dots[angle=0,distance={4pt}]}]
			(1.5,0) -- (3.5,0) -- (3.5,2) -- (1.5,2) -- (1.5,0);
		\draw[draw=none,fill = white, pattern={Lines[angle=-45,distance={5pt}]}]
			(3,0) -- (5,0) -- (5,2) -- (3,2) -- (3,0);
		\draw[draw=none,fill = white, pattern={Lines[angle=45,distance={5pt}]}]
			(0,1.5) -- (2,1.5) -- (2,3.5) -- (0,3.5) -- (0,1.5);
		\draw[draw=none,fill = white, pattern={Lines[angle=0,distance={4pt}]}]
			(1.5,1.5) -- (3.5,1.5) -- (3.5,3.5) -- (1.5,3.5) -- (1.5,1.5);
		\draw[draw=none,fill = white, pattern={Lines[angle=45,distance={5pt}]}]
			(3,1.5) -- (5,1.5) -- (5,3.5) -- (3,3.5) -- (3,1.5);
		\draw[draw=none,fill = white, pattern={Lines[angle=-45,distance={5pt}]}]
			(0,3) -- (2,3) -- (2,5) -- (0,5) -- (0,0);
		\draw[draw=none,fill = white, pattern={Dots[angle=0,distance={4pt}]}]
			(1.5,3) -- (3.5,3) -- (3.5,5) -- (1.5,5) -- (1.5,3);
		\draw[draw=none,fill = white, pattern={Lines[angle=-45,distance={5pt}]}]
			(3,3) -- (5,3) -- (5,5) -- (3,5) -- (3,3);
		\draw[black,very thick,dotted] (0,1.5) -- (5,1.5);
		\draw[black,very thick,dotted] (0,2)   -- (5,2);
		\draw[black,very thick,dotted] (0,3)   -- (5,3);
		\draw[black,very thick,dotted] (0,3.5) -- (5,3.5);
		\draw[black,very thick,dotted] (1.5,0) -- (1.5,5);		
		\draw[black,very thick,dotted] (2,0)   -- (2,5);
		\draw[black,very thick,dotted] (3,0)   -- (3,5);
		\draw[black,very thick,dotted] (3.5,0) -- (3.5,5);
		\node[anchor=north] at (0,   0) {$0.$};
		\node[anchor=north] at (1.5, 0) {$.3$};		
		\node[anchor=north] at (2,   0) {$.4$};
		\node[anchor=north] at (3,   0) {$.6$};
		\node[anchor=north] at (3.5, 0) {$.7$};		
		\node[anchor=north] at (5,   0) {$1.$};
		\node[anchor=east] at (0, 0)    {$0.$};
		\node[anchor=east] at (0, 1.5)  {$.3$};		
		\node[anchor=east] at (0, 2)    {$.4$};
		\node[anchor=east] at (0, 3)    {$.6$};
		\node[anchor=east] at (0, 3.5)  {$.7$};				
		\node[anchor=east] at (0, 5)    {$1.$};		
		\end{tikzpicture}
		\caption{$N_p = 9$ overlapping patches.}
	\end{subfigure}
	\caption{Visualization of patch overlap for the domain $\Omega = (0,1)\times (0,1)$.}\label{fig:patch_overlap}
\end{figure*}

\begin{table}[ht]
    \centering
    \caption{Objective values $J(x)$ and broken down to $f(x)$ and $\TV(x)$ and run times $t_x$ for $x = x_s$ (solution returned by SLIP)
	and $x = x_{bs}$ (solution returned by block-SLIP) for the different instances of our two-dimensional benchmark problem.
	In each row, the winner(s) in terms of objective and run time up to the reported precision are
	highlighted with bold-faced text. The improvements in run time by block-SLIP with $N_p = 9$
	for the three most expensive instances (for SLIP) are additionally highlighted with color boxes.
	Run times are reported in seconds.}\label{tab:exp2}    
	\begin{adjustbox}{width=\textwidth}	    
    \begin{tabular}{lllllllllll}
        \toprule
         $N$ & $N_{\rm p}$ & $\alpha\cdot 10^{-3}$ & $J(x_{\rm bs})$ & $J(x_{\rm s})$ & $f(x_{\rm bs})$ & $\TV(x_{\rm bs})$ & $f(x_{\rm s})$ & $\TV(x_{\rm s})$ & $t_{\rm bs}$ & $t_{\rm s}$ \\
        \midrule
        \multirow[t]{8}{*}{64} & \multirow[t]{4}{*}{4} 
            & 0.50 & \textbf{0.6749} & \textbf{0.6749} & 0.6740 & 1.8593 & 0.6740 & 1.8593 & \textbf{42} & 55 \\
         &  & 0.75 & \textbf{0.6753} & \textbf{0.6753} & 0.6740 & 1.7500 & 0.6740 & 1.7500 & \textbf{51} & 107 \\
         &  & 1.00 & \textbf{0.6758} & \textbf{0.6758} & 0.6745 & 1.2812 & 0.6745 & 1.2812 & \textbf{214} & 353 \\
         &  & 1.25 & \textbf{0.6761} & \textbf{0.6761} & 0.6745 & 1.2656 & 0.6745 & 1.2656 & \textbf{222} & 256 \\
         &  & 1.50 & \textbf{0.6764} & \textbf{0.6764} & 0.6745 & 1.2656 & 0.6745 & 1.2656 & 172 & \textbf{164} \\
         &  & 1.75 & \textbf{0.6768} & \textbf{0.6768} & 0.6745 & 1.2656 & 0.6745 & 1.2656 & \textbf{108} & 149 \\
         &  & 2.00 & \textbf{0.6771} & \textbf{0.6771} & 0.6746 & 1.2500 & 0.6746 & 1.2500 & \textbf{117} & 123 \\
         &  & 2.25 & \textbf{0.6774} & \textbf{0.6774} & 0.6746 & 1.2343 & 0.6746 & 1.2343 & \textbf{115} & 149 \\
        \cline{2-11}
         & \multirow[t]{4}{*}{9} 
            & 0.50 & \textbf{0.6749} & \textbf{0.6749} & 0.6740 & 1.8593 & 0.6740 & 1.8593 & 114 & \textbf{55} \\
         &  & 0.75 & \textbf{0.6753} & \textbf{0.6753} & 0.6740 & 1.7500 & 0.6740 & 1.7500 & 115 & \textbf{107} \\
         &  & 1.00 & \textbf{0.6758} & \textbf{0.6758} & 0.6741 & 1.7187 & 0.6745 & 1.2812 & \textbf{76} & 353 \\
         &  & 1.25 & \textbf{0.6762} & \textbf{0.6761} & 0.6741 & 1.6718 & 0.6745 & 1.2656 & \textbf{95} & 256 \\
         &  & 1.50 & 0.6766 & \textbf{0.6764} & 0.6741 & 1.6718 & 0.6745 & 1.2656 & \textbf{97} & 164 \\
         &  & 1.75 & \textbf{0.6768} & \textbf{0.6768} & 0.6745 & 1.2656 & 0.6745 & 1.2656 & \textbf{138} & 149 \\
         &  & 2.00 & \textbf{0.6771} & \textbf{0.6771} & 0.6746 & 1.2500 & 0.6746 & 1.2500 & \textbf{119} & 123 \\
         &  & 2.25 & \textbf{0.6774} & \textbf{0.6774} & 0.6746 & 1.2343 & 0.6746 & 1.2343 & \textbf{148} & 149 \\
        \cline{1-11}
        \multirow[t]{8}{*}{96} & \multirow[t]{4}{*}{4} 
            & 0.50 & \textbf{0.6711} & \textbf{0.6711} & 0.6702 & 1.8645 & 0.6702 & 1.8854 & 264 & \textbf{234} \\
         &  & 0.75 & \textbf{0.6716} & \textbf{0.6716} & 0.6702 & 1.7604 & 0.6702 & 1.7604 & \textbf{503} & 980 \\
         &  & 1.00 & \textbf{0.6720} & \textbf{0.6720} & 0.6703 & 1.7395 & 0.6707 & 1.2708 & \textbf{612} & 9151 \\
         &  & 1.25 & \textbf{0.6723} & \textbf{0.6723} & 0.6707 & 1.2604 & 0.6707 & 1.2604 & \textbf{49990} & 66740 \\
         &  & 1.50 & \textbf{0.6726} & \textbf{0.6726} & 0.6707 & 1.2604 & 0.6707 & 1.2604 & 10807 & \textbf{5628} \\
         &  & 1.75 & \textbf{0.6729} & \textbf{0.6729} & 0.6707 & 1.2395 & 0.6707 & 1.2500 & 3894 & \textbf{3733} \\
         &  & 2.00 & \textbf{0.6732} & \textbf{0.6732} & 0.6707 & 1.2395 & 0.6707 & 1.2395 & 1942 & \textbf{1614} \\
         &  & 2.25 & \textbf{0.6735} & \textbf{0.6735} & 0.6707 & 1.2395 & 0.6707 & 1.2395 & \textbf{1050} & 1228 \\
        \cline{2-11}
         & \multirow[t]{4}{*}{9} 
            & 0.50 & \textbf{0.6711} & \textbf{0.6711} & 0.6702 & 1.8645 & 0.6702 & 1.8854 & 323 & \textbf{234} \\
         &  & 0.75 & \textbf{0.6716} & \textbf{0.6716} & 0.6702 & 1.7604 & 0.6702 & 1.7604 & \textbf{434} & 980 \\
         &  & 1.00 & \textbf{0.6720} & \textbf{0.6720} & 0.6703 & 1.7395 & 0.6707 & 1.2708 & \hl{\strut\textbf{362}} & \hl{\strut 9151} \\
         &  & 1.25 & \textbf{0.6723} & \textbf{0.6723} & 0.6707 & 1.2604 & 0.6707 & 1.2604 & \hl{\strut \textbf{9017}} & \hl{\strut 66740} \\
         &  & 1.50 & 0.6728 & \textbf{0.6726} & 0.6703 & 1.6875 & 0.6707 & 1.2604 & \textbf{489} & 5628 \\
         &  & 1.75 & 0.6733 & \textbf{0.6729} & 0.6704 & 1.6562 & 0.6707 & 1.2500 & \textbf{434} & 3733 \\
         &  & 2.00 & \textbf{0.6732} & \textbf{0.6732} & 0.6707 & 1.2395 & 0.6707 & 1.2395 & \textbf{820} & 1614 \\
         &  & 2.25 & \textbf{0.6735} & \textbf{0.6735} & 0.6707 & 1.2395 & 0.6707 & 1.2395 & \textbf{586} & 1228 \\
        \cline{1-11}
		\multirow[t]{2}{*}{128} & \multirow[t]{4}{*}{4} 
		    & 1.00 & \textbf{0.6738} & \textbf{0.6738} & 0.6725 & 1.2734 & 0.6725 & 1.2734 & 393184 & \textbf{134313} \\
        \cline{2-11}
		& \multirow[t]{4}{*}{9}
            & 1.00 & \textbf{0.6738} & \textbf{0.6738} & 0.6721 & 1.7344 & 0.6725 & 1.2734 & \hl{\strut \textbf{1053}} & \hl{\strut 134313} \\        
        \bottomrule
	\end{tabular}
	\end{adjustbox}
\end{table}

\begin{table}[ht]
	\centering
	\caption{Mean and median solution times (seconds) of Gurobi for \eqref{eq:trp} instances
		per patch for $N_{\rm p} = 4$ and $\alpha = 1.00 \cdot 10^{-3}$ over the course of
		\cref{alg:slipsub_greedy}.}\label{tab:exp2cum4}	
	\begin{adjustbox}{width=.55\textwidth}	    
		\begin{tabular}{llrrrr}
			\toprule
			$N$ & & 1 & 2 & 3 & 4 \\
			\midrule
			\multirow[t]{8}{*}{64} 
			& mean   & 4.50 & 0.07 & 0.12 & 0.09 \\
			& median & 3.82 & 0.06 & 0.08 & 0.10 \\
			\cline{1-6}
			\multirow[t]{8}{*}{96} 
			& mean   & 6.77 & 0.20 & 0.15 & 0.23 \\
			& median & 4.86 & 0.19 & 0.14 & 0.22 \\
			\cline{1-6}
		    \multirow[t]{8}{*}{128} 
			& mean  & 3477.81 &  0.40 &  0.30 &  0.31 \\
			& median  & 181.49 &  0.39 &  0.25 &  0.27 \\
			\bottomrule
		\end{tabular}
	\end{adjustbox}
\end{table}
\begin{table}[ht]
	\centering
	\caption{Mean and median solution times (seconds) of Gurobi for \eqref{eq:trp} instances
		per patch for $N_{\rm p} = 9$ and $\alpha = 1.00 \cdot 10^{-3}$ over the course of
		\cref{alg:slipsub_greedy}.}\label{tab:exp2cum9}	
	\begin{adjustbox}{width=\textwidth}	    
		\begin{tabular}{llrrrrrrrrr}
			\toprule
			$N$ & & 1 & 2 & 3 & 4 & 5 & 6 & 7 & 8 & 9 \\
			\midrule
			\multirow[t]{8}{*}{64} 
			& mean    & 0.39 &  0.05 &  0.04 &  0.06 &  0.03 &  0.04 &  0.10 &  0.04 &  0.05 \\
			& median  & 0.21 &  0.05 &  0.04 &  0.05 &  0.02 &  0.04 &  0.05 &  0.04 &  0.04 \\
			\cline{1-11}
			\multirow[t]{8}{*}{96}
            & mean    & 1.35 &  0.10 &  0.04 &  0.12 &  0.06 &  0.08 &  0.16 &  0.25 &  0.07 \\
            & median  & 0.36 &  0.09 &  0.02 &  0.08 &  0.06 &  0.08 &  0.07 &  0.07 &  0.07 \\
			\cline{1-11}
			\multirow[t]{8}{*}{128} 
			& mean    & 5.40 &  0.21 &  0.07 &  0.24 &  0.19 &  0.15 &  0.13 &  0.03 &  0.10 \\
			& median  & 2.11 &  0.19 &  0.05 &  0.27 &  0.19 &  0.16 &  0.01 &  0.02 &  0.10 \\
			\bottomrule
		\end{tabular}
	\end{adjustbox}
\end{table}

\begin{figure}
\centering
\begin{subfigure}[b]{0.31\textwidth}
	\centering
	\includegraphics[width=\textwidth]{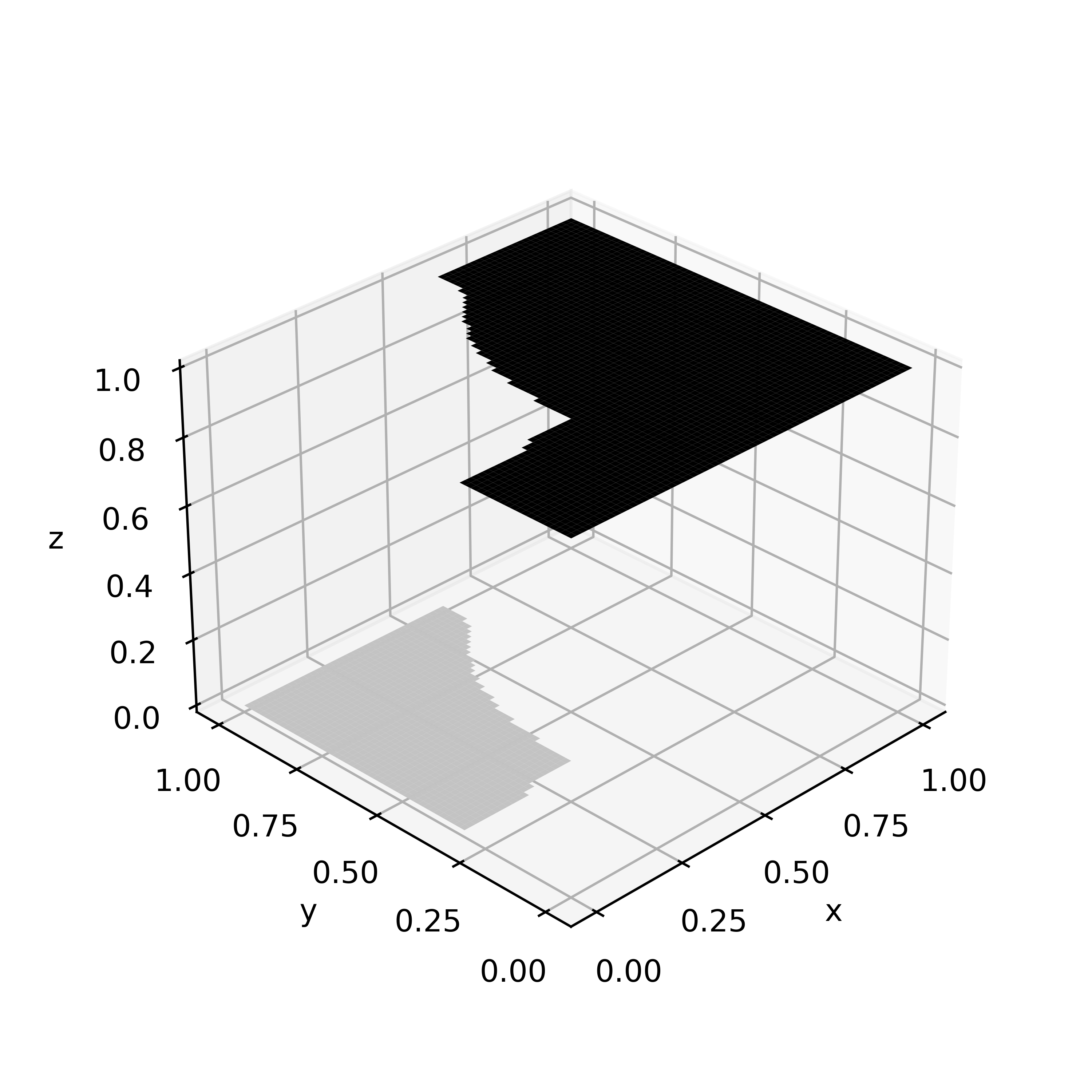}
	\caption{{\small $N = 64$, $N_{\rm p} = 1$}}
\end{subfigure}
\hfill
\begin{subfigure}[b]{0.31\textwidth}  
	\centering 
	\includegraphics[width=\textwidth]{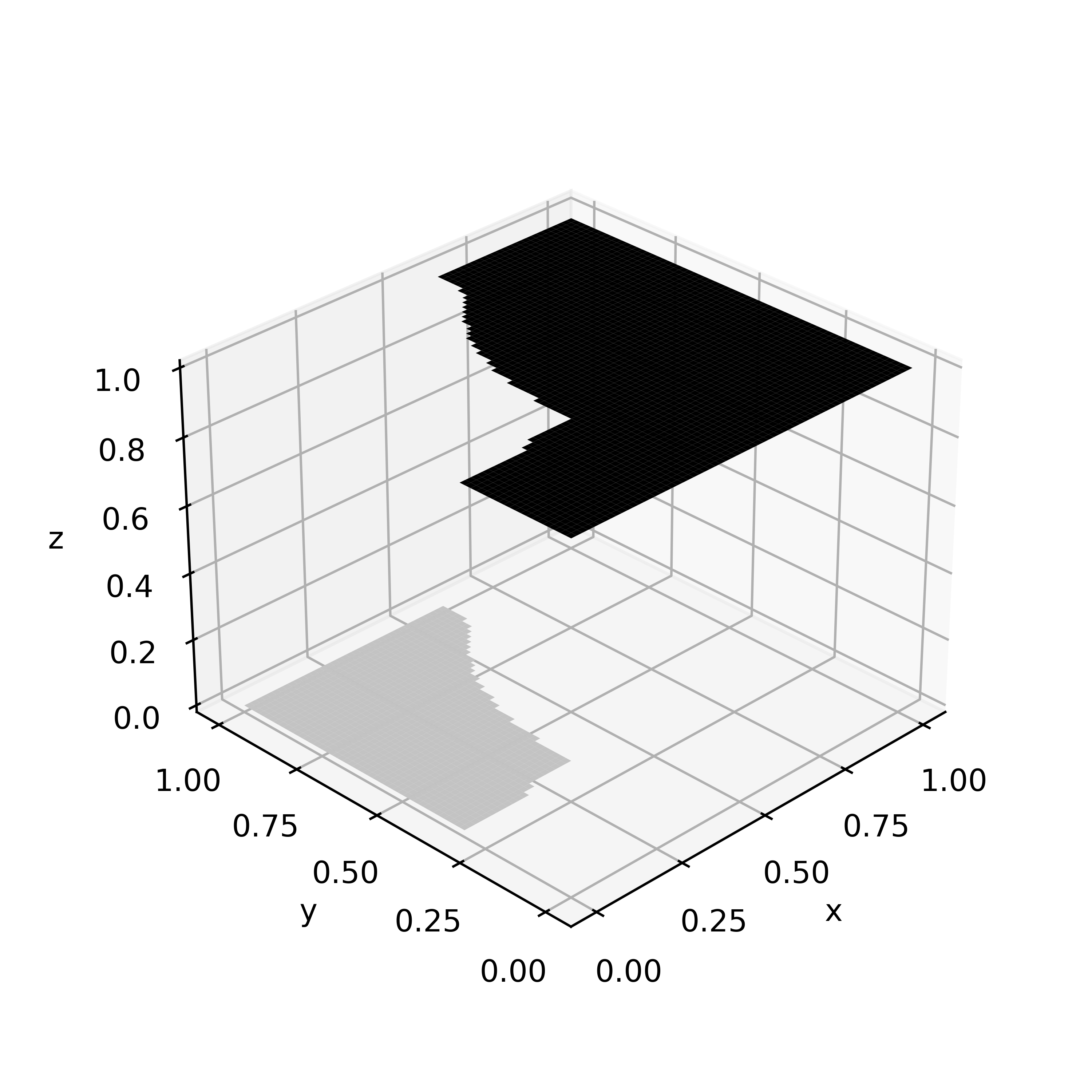}
	\caption{{\small  $N = 64$, $N_{\rm p} = 4$}}
\end{subfigure}
\hfill
\begin{subfigure}[b]{0.31\textwidth}  
	\centering 
\includegraphics[width=\textwidth]{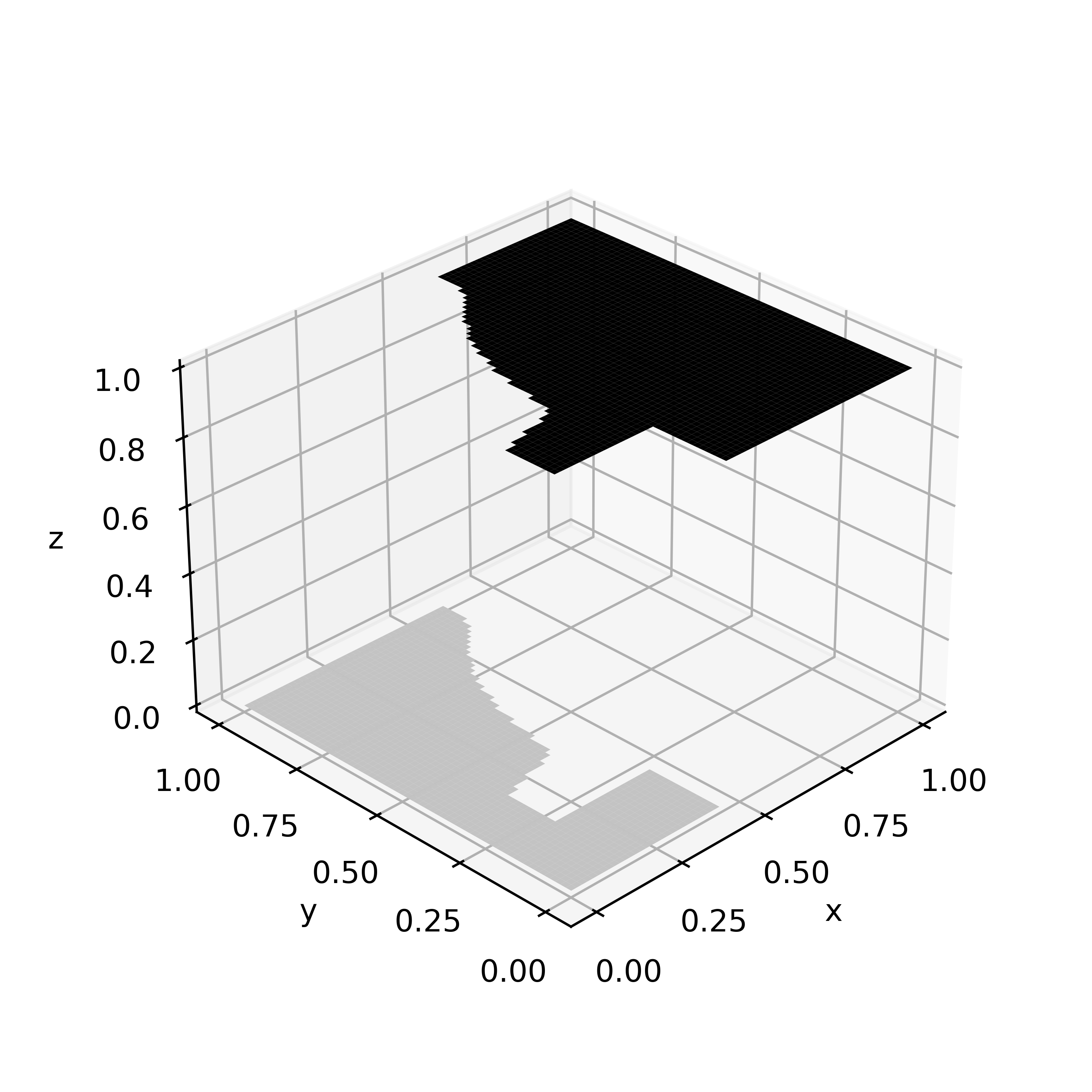}
\caption{{\small  $N = 64$, $N_{\rm p} = 9$}}
\end{subfigure}
\vskip\baselineskip
\begin{subfigure}[b]{0.31\textwidth}
	\centering
	\includegraphics[width=\textwidth]{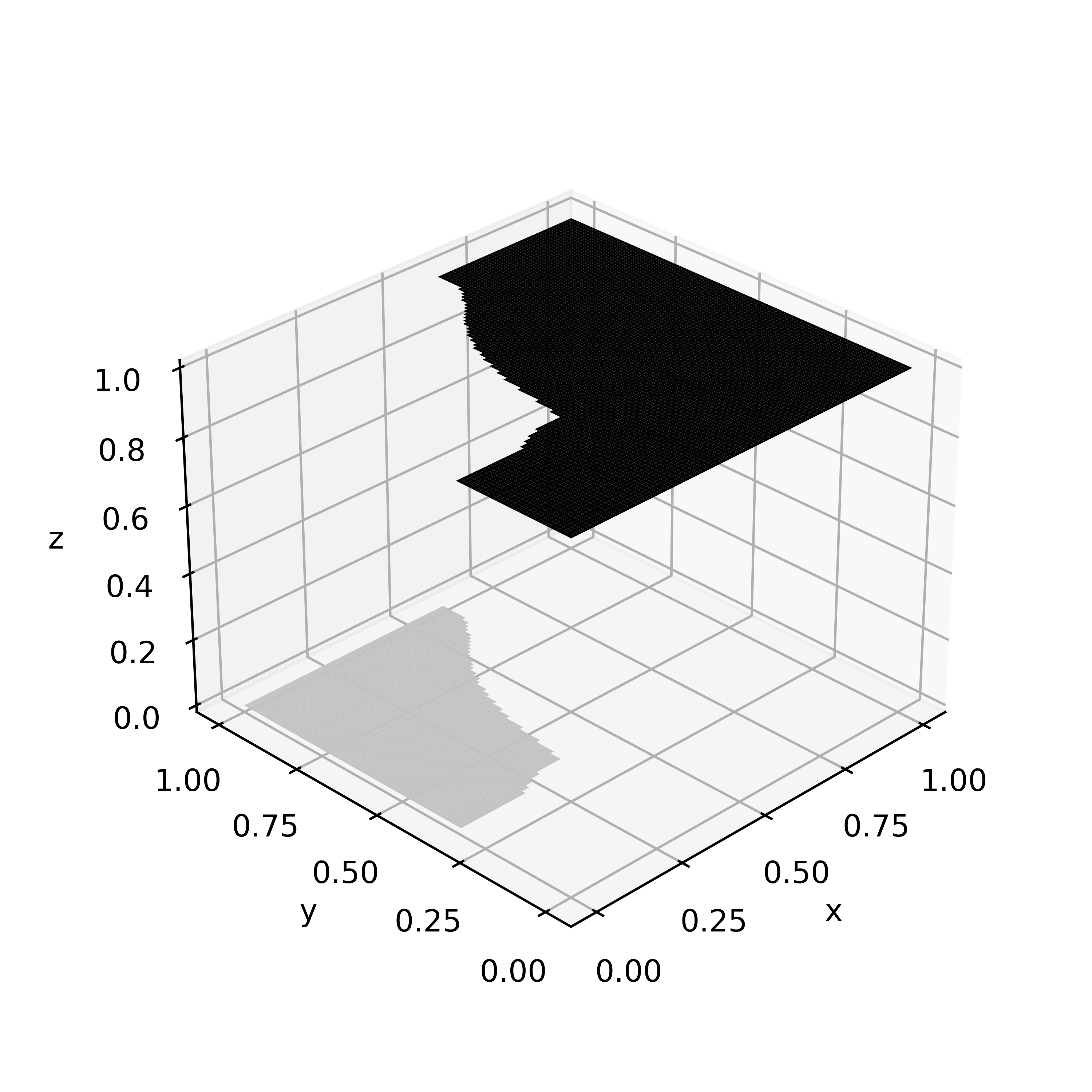}
	\caption{{\small $N = 96$, $N_{\rm p} = 1$}}
\end{subfigure}
\hfill
\begin{subfigure}[b]{0.31\textwidth}  
	\centering 
	\includegraphics[width=\textwidth]{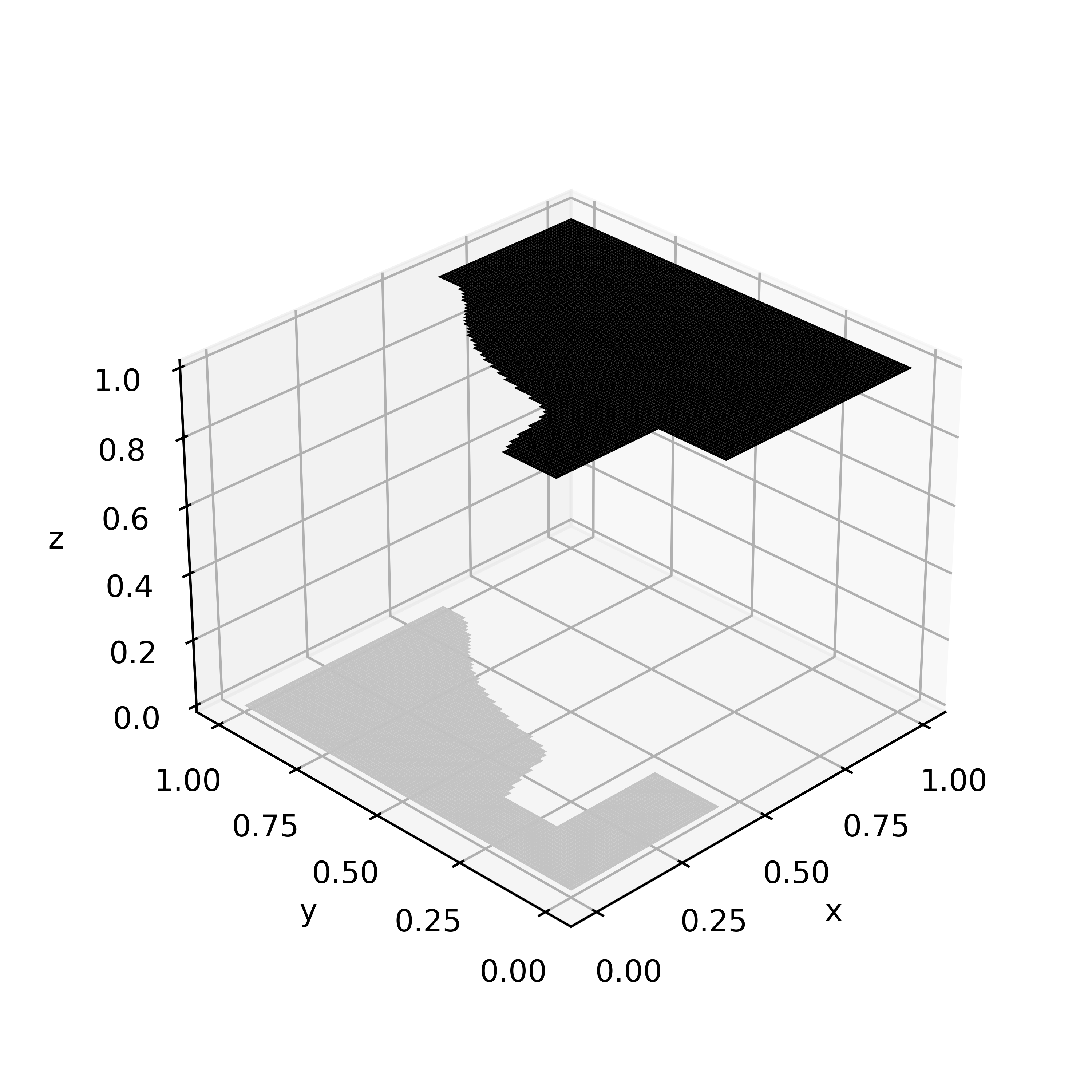}
	\caption{{\small  $N = 96$, $N_{\rm p} = 4$}}
\end{subfigure}
\hfill
\begin{subfigure}[b]{0.31\textwidth}  
	\centering 
	\includegraphics[width=\textwidth]{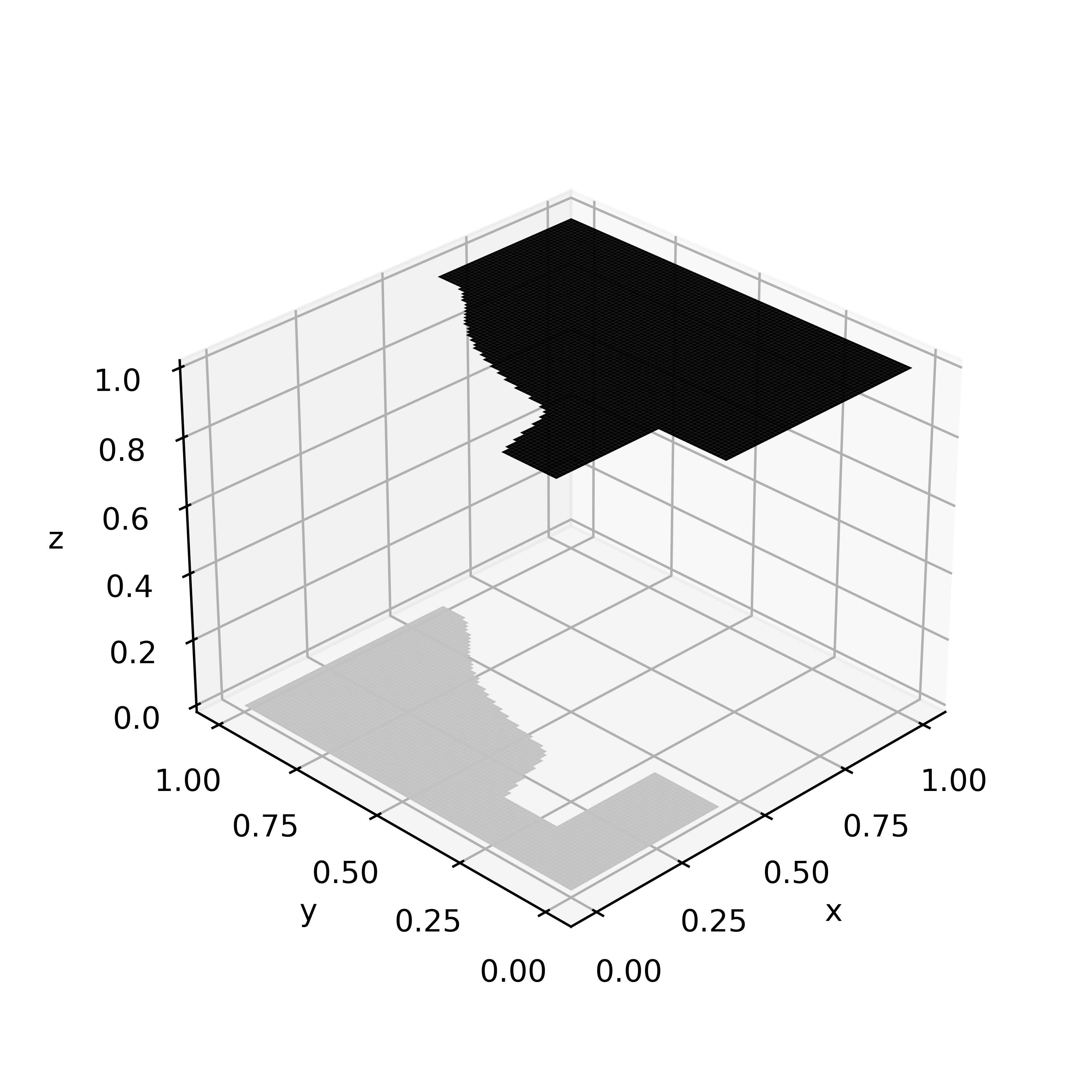}
	\caption{{\small  $N = 96$, $N_{\rm p} = 9$}}
\end{subfigure}
\vskip\baselineskip
\begin{subfigure}[b]{0.31\textwidth}
	\centering
	\includegraphics[width=\textwidth]{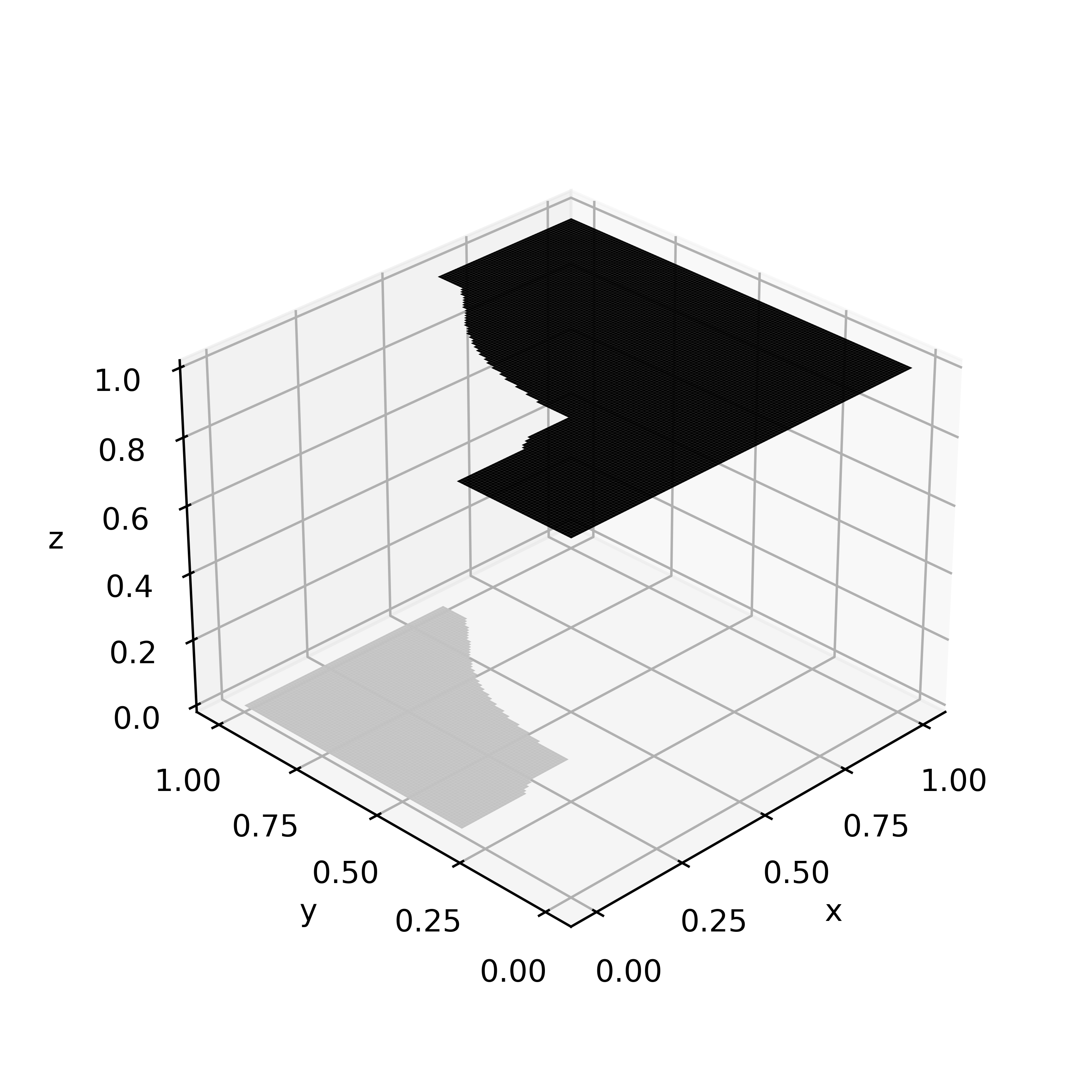}
	\caption{{\small $N = 128$, $N_{\rm p} = 1$}}
\end{subfigure}
\hfill
\begin{subfigure}[b]{0.31\textwidth}  
	\centering 
	\includegraphics[width=\textwidth]{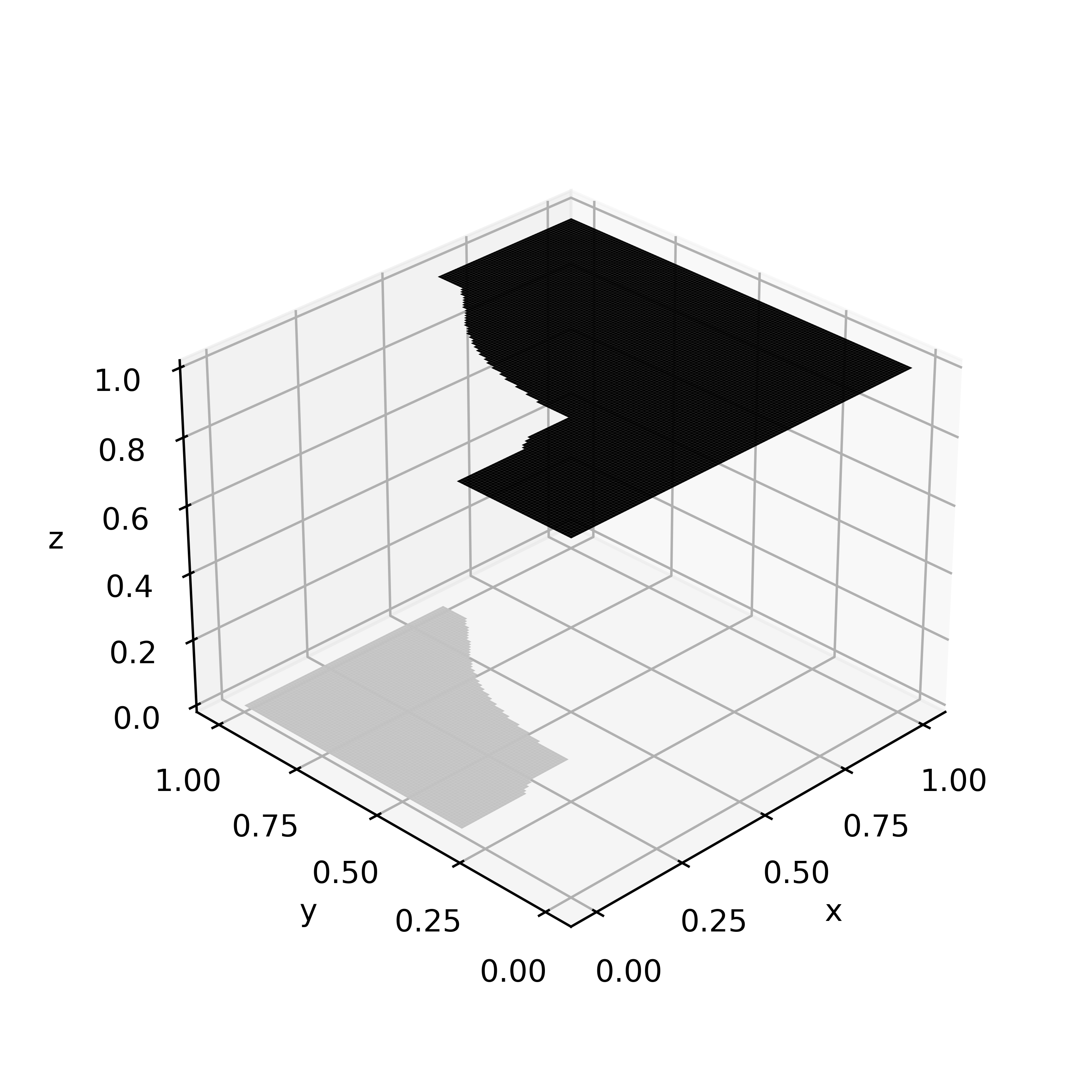}
	\caption{{\small  $N = 128$, $N_{\rm p} = 4$}}
\end{subfigure}
\hfill
\begin{subfigure}[b]{0.31\textwidth}  
	\centering 
	\includegraphics[width=\textwidth]{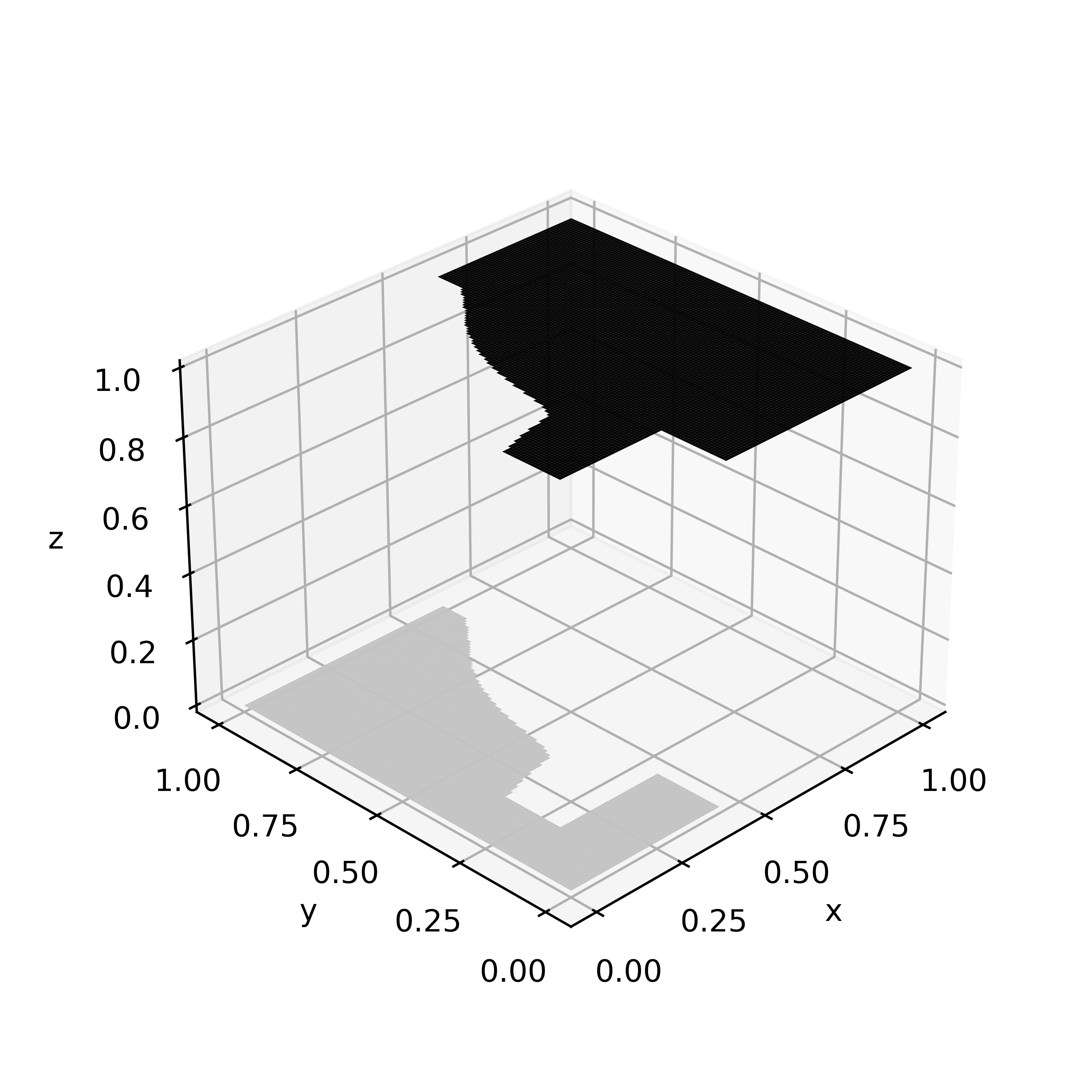}
	\caption{{\small  $N = 128$, $N_{\rm p} = 9$}}
\end{subfigure}

\caption{Visualization of the controls produced by \cref{alg:slipsub_greedy}
	for $\alpha = 10^{-3}$, $N \in \{64,96,128\}$, and $N_{\rm p} \in \{1,4,9\}$.}\label{fig:ctrl_visualization}
\end{figure}

\section*{Acknowledgments}
The authors are grateful to Francesco Maggi (UT Austin) for providing helpful guidance that lead to the competitor
constructions in the proofs of \cref{lem:instationarity,lem:weakstar_accumulation_points_are_strict}.
The authors gratefully acknowledge computing time on the
LiDO3 HPC cluster at TU Dortmund, partially funded in the Large-Scale Equipment
796 Initiative by the Deutsche Forschungsgemeinschaft (DFG) as project 271512359.
Sandia National Laboratories is a multimission laboratory
    managed and operated by National Technology and Engineering
    Solutions of Sandia, LLC., a wholly owned subsidiary of
    Honeywell International, Inc., for the U.S.\ Department of
    Energy’s National Nuclear Security Administration under
    contract DE-NA0003525.
    This paper describes objective technical results and analysis. Any
    subjective views or opinions that might be expressed in the paper
    do not necessarily represent the views of the U.S.\ Department of
    Energy or the United States Government.

\bibliographystyle{plain}
\bibliography{references}{}

\appendix

\section{Auxiliary Results}
\begin{lemma}\label{lem:symdif}
Let $A$, $B$, $C$ be sets of finite perimeter in $\Omega$,
$\Ha^{d-1}(C \cap \partial^*A) = 0$, and
$\Ha^{d-1}(C \cap \partial^*B) = 0$. Then
\[ \Ha^{d-1}\big(C \cap ((A^{(1)} \cap B^{(0)}) \cup (A^{(0)} \cap B^{(1)}))\big)
= \Ha^{d-1}\big(C \cap (A^{(1)} \symdif B^{(1)})\big).\]
\end{lemma}
\begin{proof}
Because $A$ and $B$ are sets of finite perimeter, we have
\[ \Ha^{d-1} \left(\Omega \setminus \left(A^{(0)} \cup A^{(1)} \cup \partial^* A\right)\right) = 0
\text{ and }
  \Ha^{d-1} \left(\Omega \setminus \left(B^{(0)} \cup B^{(1)} \cup \partial^* B\right)\right) = 0.
\]
Consequently, we obtain
\[
\Ha^{d-1} \left(C \cap \left(\left( A^{(1)} \cap B^{(0)} \right)
\cup \left( A^{(0)} \cap B^{(1)}\right)\right)\right)
=
\Ha^{d-1} \left(C \cap \left(\left( A^{(1)} \setminus B^{(1)} \right) \cup
\left(B^{(1)} \setminus A^{(1)}\right)\right)\right).
\]
\end{proof}

\begin{lemma}\label{lem:eq_1635}
Let $A$, $B$, $C$, $F$ be sets of finite perimeter in $\Omega$ such that
\begin{align*}
\Ha^{d-1}(\partial^*A \cap \partial^*B) = 0 &, \\
\Ha^{d-1}(\partial^*A \cap \partial^*C) = 0 &, \text{and}\\
F = (C \cap A) \cup (B\setminus A)&.
\end{align*}
Then
\[ D\chi_{F}
   = D\chi_B\mres (\Omega \setminus \overline{A}) 
   + D\chi_C\mres A
   + D\chi_A\mres (C^{(1)} \cap B^{(0)})
   - D\chi_A\mres (C^{(0)} \cap B^{(1)}).
\]
\end{lemma}
\begin{proof}
We first observe that sets of finite perimeter in $\Omega$ are also sets of finite perimeter in $\R^d$.
To distinguish the distributional derivative of a set of finite perimeter $G$ in $\R^d$ from the one
when interpreting it as a set of finite perimeter in $\Omega$, we denote the distributional derivative
of the former by $\mu_G$ as in \cite{maggi2012sets} and in contrast to $D\chi_G$ for the latter
as introduced in \cref{sec:notation}.

We apply Theorem 16.16 from \cite{maggi2012sets} and obtain from (16.35) in its proof
that
\[ \mu_F
= \mu_B \mres (\R^d \setminus \overline{A}) 
+ \mu_C \mres A
+ \mu_A \mres (C^{(1)} \cap B^{(0)})
- \mu_A \mres (C^{(0)} \cap B^{(1)}).
\]
holds when interpreting $D\chi_G$ as the distributional derivative of a set of finite perimeter
in $\R^d$. Restricting the measures on both sides to $\Omega$, we obtain
\[ \mu_F \mres \Omega
= \mu_B \mres (\Omega \setminus \overline{A}) 
+ \mu_C \mres A
+ \mu_A \mres (C^{(1)} \cap B^{(0)})
- \mu_A \mres (C^{(0)} \cap B^{(1)})
\]
because $A$, $C^{(1)} \cap B^{(0)}$, $C^{(0)} \cap B^{(1)} \subset \Omega$ already.
Since the distributional derivatives coincide in $\Omega$, we obtain $D\chi_F = \mu_F \mres \Omega$
and similar for the other terms so that the claim follows.
\end{proof}
\end{document}